\documentclass[12pt]{article}
\usepackage{amssymb,amsmath,amsfonts,amsthm}
\usepackage{mathrsfs}
\usepackage{graphicx}
\usepackage{authblk}
\usepackage{cite}
\usepackage[left = 1in,right=1in,top=1in,bottom=1in]{geometry}
\usepackage{url}
\usepackage{multirow}
\usepackage{subfigure}
\usepackage{setspace}
\usepackage{wrapfig}
\usepackage{enumerate}
\usepackage{algorithm}
\usepackage{algpseudocode}
\usepackage{hyperref}
\hypersetup{
	colorlinks   = true,    
	urlcolor     = blue,    
	linkcolor    = blue,    
	citecolor    = red,     
}
\allowdisplaybreaks

\usepackage{chngcntr}
\usepackage{apptools}
\AtAppendix{\counterwithin{lemma}{section}}
\AtAppendix{\counterwithin{theorem}{section}}
\AtAppendix{\counterwithin{definition}{section}}
\AtAppendix{\counterwithin{corollary}{section}}
\newtheorem{theorem}{{\bf Theorem}}
\newtheorem{corollary}{{\bf Corollary}}

\newtheorem{remark}{\bf Remark}
\newtheorem{definition}{{\bf Definition}}
\newtheorem{lemma}{\bf Lemma}

\DeclareMathOperator*{\argmin}{argmin}
\DeclareMathOperator*{\argmax}{argmax}

\newcommand{\defeq}{\stackrel{\bigtriangleup}{=}}

\newcommand{\RR}{\mathbb{R}}
\newcommand{\NN}{\mathbb{N}}

\newcommand{\MM}{\mathcal{M}}
\newcommand{\DD}{\mathcal{D}}
\newcommand{\Mh}{{\MM}^h}
\newcommand{\PP}{\mathcal{P}}

\newcommand{\C}{\mathcal{C}}

\newcommand{\Vol}[1]{\mathrm{Vol}\left( #1 \right)}
\newcommand{\dxj}[1]{\frac{\partial}{\partial #1}}
\newcommand{\at}[2]{\left. #1 \right\rvert_{#2}}

\newcommand{\abs}[1]{\left\lvert #1 \right\rvert}
\newcommand{\hrho}{$h$-$\rho$-$\delta$ }
\newcommand{\norm}[1]{\left\| #1 \right\|}

\title{Approximating the Riemannian Metric from Point Clouds via Manifold Moving Least Squares}
\author{Barak Sober \and Robert Ravier \and Ingrid Daubechies}
\begin{document}
\maketitle
\begin{abstract}
	The approximation of both geodesic distances and shortest paths on point cloud sampled from an embedded submanifold $\MM$ of Euclidean space has been a long-standing challenge in computational geometry. Given a sampling resolution parameter $ h $, state-of-the-art discrete methods yield $ O(h) $ provable approximations. In this paper, we investigate the convergence of such approximations made by Manifold Moving Least-Squares (Manifold-MLS), a method that constructs an approximating manifold $\Mh$ using information from a given point cloud that was developed by Sober \& Levin in 2019 .
	In this paper, we show that provided that $\MM \in C^{k}$ and closed (i.e. $\MM$ is a compact manifold without boundary) the Riemannian metric of $ \Mh $ approximates the Riemannian metric of $ \MM, $.
	Explicitly, given points $ p_1, p_2 \in \MM $ with geodesic distance $ \rho_\MM(p_1, p_2) $, we show that their corresponding points $ p_1^h, p_2^h \in \Mh $ have a geodesic distance of $ \rho_\Mh(p_1^h,p_2^h) = \rho_\MM(p_1, p_2)(1 + O(h^{k-1})) $ (i.e., the Manifold-MLS is nearly an isometry).
	We then use this result, as well as the fact that $ \Mh $ can be sampled with any desired resolution, to devise a naive algorithm that yields approximate geodesic distances with rate of convergence $ O(h^{k-1}) $. 
	We show the potential and the robustness to noise of the proposed method on some numerical simulations.
\end{abstract}

\section{Introduction}
The concept of a manifold is fundamental to models in numerous research fields (e.g., mathematical physics, linear algebra, computer graphics, computer aided design). 
In recent years, manifold-based modeling has also permeated into general high dimensional data analysis as a method of circumventing the curse of dimensionality (e.g., signal processing, computer vision, biomedical imaging).
These ``\emph{data-driven}" models, however, do not follow from a set of equations or some other formal representation, but rather emerge from a general intuition/observation that data has low intrinsic dimensions, at least locally (see, for example, \cite{belkin2003laplacian,coifman2006diffusion,roweis2000LLE,saul2003think,tenenbaum2000global}).
The lack of formalism means that both explicit coordinate charts and Riemannian metric structure are usually not available; in practice, the only information available is the collection of samples, i.e., \textit{point clouds}. 

An essential tool for data analysis in modern data-driven scenarios is the ability to measure intrinsic geodesic distances.
The approximation of geodesic distances and geodesic paths on point clouds and their manifolds has been studied extensively in computer graphics literature, primarily in the context of 3D point clouds that represent 2D surfaces.
The focus of these studies is on performing efficient approximations of high approximation order (for example,  \cite{mitchell1987discrete,sethian1996fast,kimmel1998computing}).
In most of these methods there is a strong assumption that a meshing of the manifold exists, which is a nontrivial requirement for general point clouds.
To the best of our knowledge, the highest order of convergence proven in the literature is $ O(h) $ (for a resolution parameter $ h $) shown in the Fast Marching Method of Kimmel and Sethian \cite{kimmel1998computing}. Numerically, the work in \cite{lichtenstein2019deep} reports $ O(h^2) $ convergence rates in numerical simulations.
Nonetheless, much less attention have been given to the general manifold case, where a triangulation cannot be achieved easily.
Memoli and Sapiro \cite{memoli2005distance} proposed an approach for general submanifold of $ \RR^D $ with convergence rates of $ O(\sqrt{h}) $.

Many works deal with the lack of triangulation by instead using standard shortest path algorithms, such as the well-known Dijkstra algorithm \cite{dijkstra1959note}, after imposing a graph structure on the point cloud based on distance information. Though nonexhaustive, see, for example, \cite{moscovich2016minimax,tenenbaum2000global,lin2008riemannian,ravier2018algorithms,ravier2018eyes}. 
However, approximating geodesic curves and distances from discrete samples in such a way may result in poor approximations if the sampling resolution is insufficient. Figure \ref{fig:DiscreteShortestPath_Problem}. gives an example of such a failure.
As discussed in Section \ref{sec:Numerical}, in such cases, it is possible (though by no means guaranteed) that the Euclidean distance between points on the manifold serve as a better approximation to the geodesic distance.

\begin{figure}[h]
	\centering
	\includegraphics[width=\linewidth]{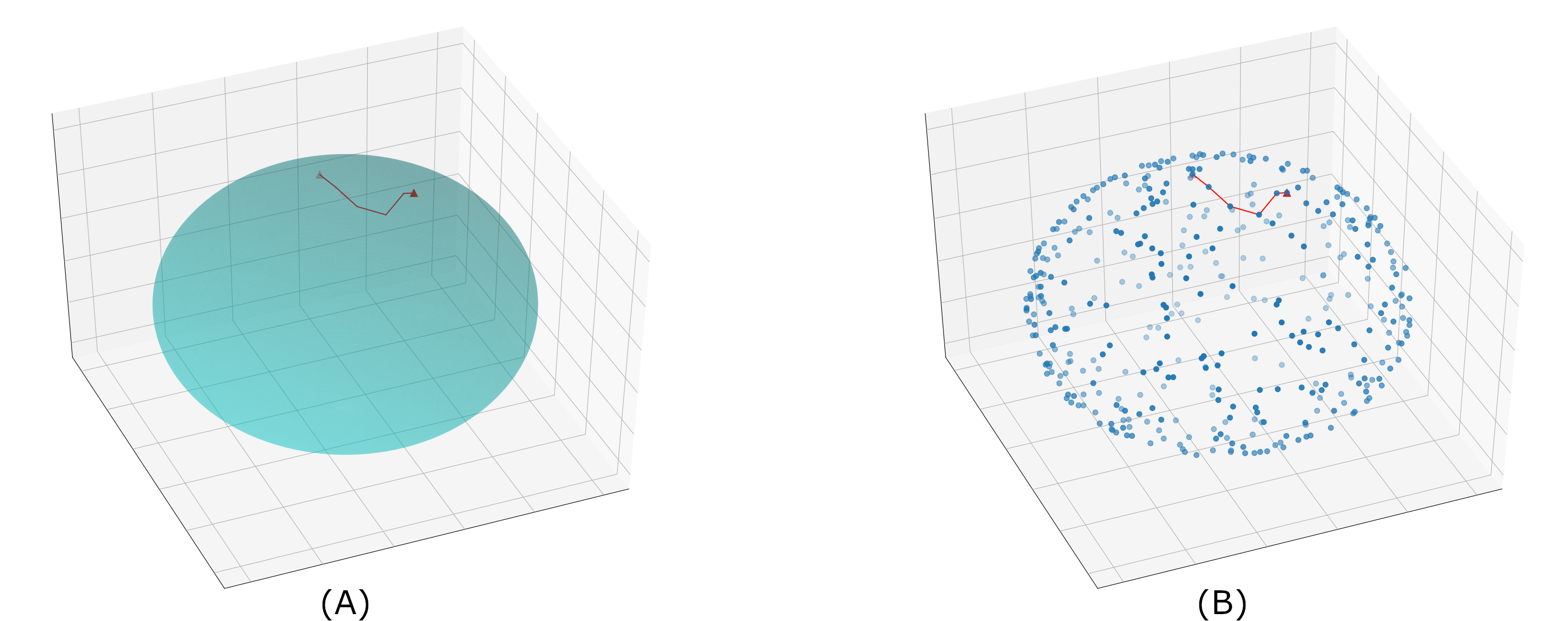}
	\caption{Shortest path from discrete samples of a 2-sphere. (A) The original manifold $ \MM $; the shortest path computed between $ p_1, p_2 $ (marked by red triangles) on discrete samples of the manifold (the path marked by the red line).
		(B) The sample set of $ \MM $ and the shortest path connecting points $ p_1, p_2 $ marked by the red triangles.}
	\label{fig:DiscreteShortestPath_Problem}
\end{figure}

The main contribution of the current work is theoretical.
In \cite{sober2019manifold}, Sober and Levin proposed a procedure called Manifold Moving Least Squares (Manifold-MLS) that defined an implicit manifold from a given point cloud.
Explicitly, given a Point Cloud sampled from a manifold $ \MM\in \C^k $ ($ k\geq 2 $) with sampling resolution $ h $ (such as the fill distance; see \eqref{def:h} below), they introduce a moving least squares (MLS) procedure to create an approximating manifold $ \Mh.$
They show that $\PP:\MM\to\Mh$, the Manifold-MLS mapping from $\MM$ to $\Mh$, is a diffeomorphism.
Similar to standard MLS approximations, the approximating manifold exists as an abstract object and can be realized at a point based upon a local neighborhood, requiring no global processing.
Sober and Levin also show that the Hausdorff distance between $ \Mh $ and $ \MM $ is bounded by $ O(h^k) $.

In this paper, we show that the Riemannian metric of $ \Mh $ approximates the Riemannian metric of $ \MM $ as well.
Specifically, let $ p_1, p_2 \in \MM $ with geodesic distance $ \rho_{\MM}(p_1, p_2) $ then
\[
	\rho_{\Mh}(p^h_1, p^h_2) = \rho_{\MM}(p_1, p_2)(1 + O(h^{k-1}))
,\]
where $ p^h_1, p^h_2\in \Mh $ are $ O(h^k) $ away from $ p_1, p_2,$  respectively, and $ \rho_{\Mh}(p^h_1, p^h_2) $ is the Riemannian distance between $ p^h_1, p^h_2 $ on $ \Mh $ (see Theorem \ref{thm:MainResult}).
We prove this by showing that $\PP$ is actually an approximate isometry given that $\MM$ is sampled densely enough. In other words, the differential of $\PP$ is close to the identity map.
In Section \ref{sec:Numerical}, we show how this result can be utilized by developing a simple brute-force algorithm that significantly improves the estimation of geodesics and their distances.
The proposed algorithm can be shown to converge with a rate of $ O(h^{k-1}),$ a significant improvement on other results.
We provide experimental evidence showing the potential of using the Manifold-MLS as well as its robustness to noise.
We believe that the current work will motivate the future development of novel, efficient ways of approximating geodesics on point clouds

\section{Definitions, Assumptions and the Approximation}\label{sec:Preliminaries}
The analysis carried in Section \ref{sec:proof} builds upon the definition of the Manifold-MLS as introduced in \cite{sober2019manifold}, which culminated in a bound on the difference between the approximated manifold $ \MM $ and its approximant with respect to the Hausdorff distance.
However, since we wish to analyze the convergence of geodesic paths, we ultimately need to bound the distortion of first derivatives as well.
Therefore, we will start by explaining some recent results in derivative estimation of standard MLS and carry on to formally define the Manifold-MLS approximation along with the sampling assumptions that will be used in the proofs below.
\subsection{Moving Least-Squares Approximation of Derivatives}\label{sec:MLS_Derivative}
A key concept we aim to use in our analysis is the approximation property of the natural derivatives of the MLS have as reported in \cite{mirzaei2015analysis}.

Let $ f:\Omega\to\RR $ be sampled at locations $ X = \{x_i\}_{i=1}^n $.
Then, we define the local polynomial $ \pi^*(x ~|~\xi) $ by 
\begin{equation}\label{eq:MLS_pi_take1}
\pi^*(x ~|~ \xi) = \argmin_{\pi\in\Pi_{k-1}^d}\sum_{i=1}^n\abs{f(x_i) - \pi(x_i)}^2\theta(\norm{\xi - x_i})
,\end{equation}
where $ \Pi_{k-1}^d $ is the space of polynomial of degree $ \leq (k-1) $ over a $ d $-dimensional domain, $ \theta $ is a locally supported (or fast decaying) weight function, and $ \norm{\cdot} $ denotes the Euclidean norm.
We define the MLS function by
\begin{equation}\label{eq:MLS_basic_pre}
s^{\text{MLS}}_{f,X}(x) \defeq \pi^*(0 ~|~ x) 
.
\end{equation}

The theory developed in \cite{mirzaei2015analysis} shows that the $ \alpha $-order derivatives of $ s^{\text{MLS}}_{f,X}(x) $ approximate the derivatives of $ f(x) $ with convergence rates of $ O(h^{k-\abs{\alpha}}) $ for $ f , \theta \in \C^k $ and $ \abs{\alpha} \leq k-1 $. 
However, the analysis of \cite{mirzaei2015analysis} is described in a more general setting, where the approximation order is derived for arbitrary fractional Sobolev norm and with a domain with Lipschitz boundary.
Since we are interested in the ordinary derivatives of the MLS approximation (more precisely, just first order derivatives) in a boundaryless domain, in this subsection we describe and justify the main result that we wish to use in this paper, which is presented in \eqref{eq:DerivativeApprox}: the MLS approximant's first order natural derivatives have optimal convergence rates to $ f $'s derivatives.

Let $ X = \{x_i\}_{i=1}^n $ be samples of a domain $ \Omega \subset\RR^d$.
Furthermore, we assume that the samples are a quasi-uniform sample set.
\begin{definition}[quasi-uniform sample set]\label{def:quasi-uniform}
	A set of data sites $ X = \{x_1,\ldots,x_n\} $ is said to be quasi-uniform with respect to a domain $ \Omega $ and a constant $ c_{\emph{qu}} > 0 $ if
	\[
	\delta_X \leq h_{X,\Omega} \leq c_{\emph{qu}}\delta_X
	,\]
	where $ h_{X,\Omega} $ is the fill distance
	\begin{equation}\label{def:h}
	h_{X,\Omega} = \sup_{x\in\Omega} \min_{x_i \in X} \norm{x - x_i}
	,\end{equation}
	and $ \delta_X $ is the separation radius defined by
	\[
	\delta_X := \frac{1}{2}\min_{i\neq j}\norm{x_i - x_j}
	.\]
	Throughout the paper we will denote for short $ h:=h_{X,\Omega} $. 
\end{definition}

Then, as described in \eqref{eq:MLS_pi_take1}, we define the MLS local polynomial $ \pi^*(x ~|~ \xi) $ as
\begin{equation}\label{eq:MLS_basic_Minimization}
\pi^*(x ~|~ \xi) = \argmin_{\pi\in\Pi_{k-1}^d}\sum_{i=1}^n\abs{f(x_i) - \pi(x_i)}^2\theta_h(\norm{\xi - x_i})
,\end{equation}
where now $ \theta_h $ is a locally supported weight function consistent across scales (i.e., $ \theta_h(th) = \Phi(t) $).
Furthermore, we assume that $ \theta_h\in\C^k $ and has a compact support of size $ sh $, and we denote by $ I(x) = \{i ~|~ \norm{x - x_i}\leq sh, x_i\in X  \} $.
Levin showed in \cite{levin1998approximation} that, given enough samples in $ I(x) $ (i.e, if the least-squares matrix is invertible), the MLS approximation defined in  \eqref{eq:MLS_basic_pre} is given by
\begin{equation}\label{eq:s_fX_MLS_pre}
s^{\text{MLS}}_{f,X}(x) = \sum_{i\in I(x)} a^*_i(x) f(x_i) 
,
\end{equation}
where the coefficients $ a^*_i(x) $ are determined by minimizing the quadratic form
\begin{equation}\label{eq:MLS_Basic_Quadratic_Form}
 \sum_{i\in I(x)}a_i(x)^2\frac{1}{\theta_h(\norm{x-x_i})}
\end{equation}
under the constraints 
\[
\sum_{i\in I(x)}a_i(x) \pi(x_i) = \pi(x),~~ \forall \pi\in \Pi_{k-1}(\RR^d)
.\]

In addition, it has been shown that $ s_{f,X}^\text{MLS}(x)\in \C^k $  \cite{levin1998approximation,wendland2004scattered}, and that it has high approximation order in both the $ L_\infty $ and Sobolev sense (see \cite{levin1998approximation,wendland2004scattered} for the $ L_\infty $ analysis,  \cite{armentano2001error} for a first order Sobolev, and \cite{mirzaei2015analysis} for arbitrary fractional Sobolev norm).
The key idea behind these results stems from the fact that $ s_{f,X} $ is exact for polynomials.
Namely,
\[
\sum_{i\in I(x)}a_i^*(x) \pi(x_i) = \pi(x)
\]
for $ \pi\in \Pi_{k-1}(\RR^d) $.
Thus,
\begin{align*}
E_{f, X}(x) &= s_{f, X} - f(x)\\ 
&= s_{f, X} - \pi(x) + \pi(x) - f(x) \\
&=  \sum_{i\in I(x)}a_i^*(x) (f(x_i) - \pi(x_i)) + \pi(x) - f(x)
.\end{align*}
Then, for a fixed but arbitrary $\hat x\in \Omega$ we can choose $\pi$ to be the Taylor approximation of $f(x)$ around $\hat x$ and then we get for all $ x\in B_{sh}(\hat x) $
\begin{align*}
\abs{E_{f, X, B_{sh}(\hat x)}(\hat x)} 
&\leq\sum_{i\in I(x)}\abs{a_i^*(x)}\abs{ (f(x_i) - \pi(x_i))} + \abs{\pi(x) - f(x)} \\
&\leq \max_{x\in B(\hat x, sh)}\abs{\pi(x) - f(x)} \left(\sum_{i\in I(\hat x)}\abs{a_i^*(x)}+1\right) 
,\end{align*}
where the number of indices in $ I(x) $ are bounded by $ N = \rho s^d $ for some constant $ \rho $ (namely, the bound is independent of $ h $; see Lemma \ref{lem:h-delta-is-rho} in the Appendix). Since the $ a_i^* $ depend smoothly on the location of their neighbors in a $ sh $-sized neighborhood, the summation can be bounded by a constant independent of $ h, x,$ and the specific sample set $ X.$ See the Appendix for a more rigorous discussion of this fact.
Since we chose $ \pi $ to be the $ (k-1) $ degree Taylor expansion of $ f(x) $ we get
\begin{align}\label{eq:TaylorBound}
\max_{x\in B(\hat x, sh)}\abs{f(x) - \pi(x)} &\leq c \abs{f}_{\C^k(B_{sh}(\hat x))} h^k 
,\end{align}
where $ \abs{f}_{\C^k(B_{sh}(\hat x))}  $ is the semi-norm defined by   
\begin{equation}
 \abs{f}_{\C^k(B_{sh}(\hat x))} \defeq \max_{\begin{subarray}{c}
 	\abs{\alpha} = k\\ x\in B_{sh}(\hat x)
 	\end{subarray} } \abs{ \partial^\alpha f(x)}
,\end{equation}
where $ \alpha = (\alpha_1, \ldots, \alpha_d) $ is a multi-index of order $ \abs{\alpha} = k $ and 
\[
\partial^\alpha f(x) = \partial^{\alpha_1}_{x_1} \partial^{\alpha_2}_{x_2} \cdots \partial^{\alpha_d}_{x_d}f(x)
.\]
As a result, we achieve the bound
\begin{equation}
	\abs{s_{f, X}(x) - f(x)} = \abs{E_{f, X, B_{sh}(\hat x)}(x)} \leq C \abs{f}_{\C^k(B_{sh}(\hat x))} h^k
.\end{equation}

If, on the other hand, one wishes to measure the error in the derivatives we get
\[
\partial^\alpha E_{f, X, \Omega}(x) = \sum_{i\in I(x)}(\partial^\alpha a_i^*(x)) (f(x_i) - \pi(x_i)) + \partial^\alpha(\pi(x) - f(x))
.\]
Note that in the current case, $\partial^\alpha$ is scalar-valued as $f:\RR^d\to \RR$.
Then, for a fixed but arbitrary $\hat x\in \Omega$ we can choose $\pi$ to be the Taylor approximation of $f(x)$ around $\hat x$ and get
\begin{equation}\label{eq:Err_Norm_Scalar}
\abs{\partial^\alpha E_{f, X}(\hat x)} \leq \max_{x\in B(\hat x, sh)}\abs{f(x) - \pi(x)}\sum_{i\in I(\hat x)} \abs{\partial^\alpha a_i^*(\hat x)} + \abs{\partial^\alpha \pi(x) - \partial^\alpha f(x)}
.\end{equation}
Since $ \pi $ is the Taylor expansion of $ f(x) $ at $ \hat x $ and $ \norm{\hat x - x} \leq sh $ we get that
\[
\abs{\partial^\alpha \pi(x) - \partial^\alpha f(x)} \leq c_1 \abs{f}_{\C(B_{sh}(\hat x))} h^{k - \abs{\alpha}}
.\]
Then, using arguments similar to Theorem 3.11 in \cite{mirzaei2015analysis} we can achieve for all $ x\in B_{sh}(\hat x) $
\begin{equation}\label{eq:a_i_deriv_order}
    \abs{\partial^\alpha a_i^*(x)} \leq c_2 h^{-\abs{\alpha}}
.\end{equation}
We can combine \eqref{eq:TaylorBound} and the above bounds with the fact that the number of indices in $ I(x) $ is bounded by $ N=\rho s^d $ to get the following lemma.
\begin{lemma}\label{lem:MLS_Deriv_scalar}
Let $f\in \C^k(\RR^d)$ be a scalar-valued function. Let $X\subset \Omega$ be a quasi-uniform unbounded set. 
Suppose $\theta_h\in \C^{k}$ is supported on $[0, sh]$.
Then, if for a fixed but arbitrary $\hat x\in \Omega$ the set $X$ is $\Pi_{k-1}(\RR^d)$-unisolvent \emph{(}i.e., the least-squares matrix resulting from \eqref{eq:MLS_basic_Minimization} is invertible\emph{)}, we get for all $ x\in B_{sh}(\hat x) $
\begin{equation}\label{eq:DerivativeApprox_scalar}
\abs{\partial^\alpha s_{f, X}(\hat x) - \partial^\alpha f(\hat x)} = \abs{\partial^\alpha E_{f, X}(\hat x)}\leq C \abs{f}_{\C^k(\hat x)} h^{k - \abs{\alpha}}
,\end{equation}
where $C$ is some constant independent of $f$ and $h$.
\end{lemma}

\begin{remark}
    Lemma \ref{lem:MLS_Deriv_scalar} is adapted from Theorem 3.12 of \cite{mirzaei2015analysis}, which handles a much more general case.
\end{remark}

\begin{remark}
	Wendland invest a considerable effort in showing what is the tightest support for $ \theta_h $ such that the MLS problem has a unique solution (i.e., that the least-squares matrix is invertible).
	In order to achieve that he introduces the Interior Cone Condition \cite{wendland2004scattered}, which is then inherited into the work of Mirzaei \cite{mirzaei2015analysis}.
	The main challenge in these investigations is the fact that the domain has boundary.
	Nevertheless, in the work of Levin \cite{levin1998approximation} and the following Manifold-MLS works \cite{sober2019manifold,sober2017approximation} as well as in the current paper the domain is without boundary, and thus, the Interior Cone Condition is not necessary (see the Appendix for a more detailed discussion).
\end{remark}

Although the above discussion was conducted on scalar-valued function $f$, a similar line of arguments will suit the case where $f:\RR^d\to \RR^D$.
Note that the coefficients $a_i^*(x)$ depend only on the sites of the sample set $X$ and not on the values.
Thus, $a_i^*(X)$ are the same for all $D$ coordinates in the target domain, and we can write
\[
\partial^\alpha E_{f, X}(x) = \sum_{i\in I(x)}(\partial^\alpha a_i^*(x)) (f(x_i) - \pi(x_i)) + \partial^\alpha(\pi(x) - f(x))
,\]
where $a_i^*(x):\RR^d\to\RR$, and $\pi:\RR^d\to\RR^D$.
Then, the error given in \eqref{eq:Err_Norm_Scalar} for a fixed but arbitrary $\hat x \in \Omega$ and $\pi$ being the coordinate-wise Taylor expansion of $f$ around $\hat x$ can be written for the $\ell^\textrm{th}$  coordinate ($1\leq \ell \leq D$) as
\[
\partial^\alpha E^\ell_{f, X}(x) = \sum_{i\in I(x)}(\partial^\alpha a_i^*(x)) (f^\ell(x_i) - \pi^\ell(x_i)) + \partial^\alpha (\pi^\ell(x) -  f^\ell(x) )
.\]
Then using the same considerations on a coordinate-wise level we arrive at
\begin{lemma}\label{lem:MLS_Deriv}
Let $f:\RR^d\to\RR^D$ and $f\in \C^k$.
Let $X \subset \Omega$ be a quasi uniform unbounded set. 
Suppose $\theta_h\in \C^{k}$ is supported on $[0, sh]$ with vanishing derivatives at the end points.
Then, if for a fixed but arbitrary $\hat x\in \Omega$ the set $X$ is $\Pi_{k-1}(\RR^d)$-unisolvent \emph{(}i.e., the least-squares matrix resulting from \eqref{eq:MLS_basic_Minimization} is invertible\emph{)}, we get for all $1 \leq \ell \leq D$ and all $ x\in B_{sh}(\hat x) $
\begin{equation}\label{eq:DerivativeApprox}
\abs{\partial^\alpha s^\ell_{f, X}( x) - \partial^\alpha f^\ell( x)} = \abs{\partial^\alpha E^\ell_{f, X}( x)}\leq C \norm{f}_{\C^k(B_{sh}(\hat x))} h^{k - \abs{\alpha}}
,\end{equation}
where
\begin{equation}
\norm{f}_{\C^k(B_{sh}(\hat x))} \defeq \max_{1 \leq \ell \leq D}\abs{f^\ell}_{\C^k(B_{sh}(\hat x))}
,\end{equation}
and $C$ is some constant independent of $f$ and $h$.
\end{lemma}

\subsection{Manifold Moving Least-Squares (Manifold-MLS)}\label{sec:MMLS}
We now turn to presenting the Manifold-MLS \cite{sober2019manifold}, which generalizes the MLS to approximate submanifolds of $ \RR^D $ based upon scattered data.
Let us begin with the sampling assumptions. 

\subsubsection*{Clean Sampling Assumptions}
\label{sec:CleanSampling}
\begin{enumerate}
	\item $\MM\in \C^k$ ($ k \geq 2 $) is a closed (i.e., compact and boundaryless) submanifold of $\RR^D$.
	\item $R = \{r_i\}_{i=1}^n\subset\MM$ is a quasi-uniform sample set with respect to the domain $\MM$ (see Definition \ref{def:quasi-uniform}); we denote the fill-distance of our sample $h_{R, \MM}$ by $h$ for short.
\end{enumerate}

Given a point $r$ near $\MM$ the Manifold Moving Least-Squares (Manifold-MLS) projection of $r$ is defined through two sequential steps: 
\begin{itemize}
    \item[1.] Find a local $d$-dimensional affine space $(q(r),H(r))$ that approximates the sampled points.
    Explicitly, $H(r)$ is a linear space and the origin is set to $q(r)$ .
    So, $\{q + h | h\in H\}$ is the affine subspace that approximates the sampled points.
    \item[2.] Define the projection of $r$ using a local polynomial approximant $\pi:H\simeq\RR^d \rightarrow \mathbb{R}^{D}$ of $\mathcal{M}$ over the new coordinate system. 
    Explicitly, we denote by $x_i$ the projections of $r_i - q$ onto $H$ and then define the samples of a function $\varphi$ by $\varphi(x_i) = r_i$. 
\end{itemize}

\noindent\textbf{Step 1 - the local coordinate system} \\
Let 
\begin{equation}
   J_1(q, H ~|~ r) = \sum_{i=1}^{n} d(r_i-q , H)^2 \theta_1(\| r_i - q\|) 
,\label{eq:J1def}\end{equation}
where $ \theta_1(t) $ is a fast decaying radial weight function.
We wish to Find a $d$-dimensional affine space $H(r)$, and a point $q(r)$ on $H(r)$, such that 
\begin{equation}
   (q(r),H(r)) = \argmin_{q\in\RR^D,H\in Gr(d,D)} J_1(q, H ~|~ r) 
\label{eq:Step1Minimization}\end{equation}
under the constraints
\begin{enumerate}
\item $r-q \perp H$  \label{init_constraint:perp}
\item $q\in B_{\mu}(r)$ \label{init_constraint:search}
\item $\#\left(R\cap B_{ h}(q)\right) \neq 0$ \label{init_constraint:proximity}
,\end{enumerate}
where $d(r_i - q , H)$ is the Euclidean distance between  $r_i - q$ and the linear subspace $H$, $B_\mu(r)$ is an open ball of radius $\mu$ around $r$ limiting the Region Of Interest (ROI), and $h$ is the fill distance $ h_{\MM, R} $ of the sampling assumptions.
Note, that $ \mu $ must be limited by the injectivity radius of $ \MM $ (i.e., manifold's reach), but we omit the relevant discussions from our paper and instead refer interested readers to \cite{sober2019manifold,sober2017approximation}.

\noindent\textbf{Step 2 - the weighted least-squares} \\
Upon obtaining a local coordinate system we now pronounce the manifold in $B_\mu(r)$, the restricted ROI, as a function $\varphi:H(r)\to\RR^D$. 
The approximation in the local coordinates is performed by means of finding a polynomial $ \pi^*(x)\in\Pi_{k-1}^d $ which minimizes 
\begin{equation}\label{eq:J2def}
    J_2(\pi(x) ~|~ q, H ) = \sum_{i=1}^n \norm{\pi(x_i) - r_i}^2\theta_2(x_i) 
,\end{equation}
where $ x_i = P_{H}(r_i - q) $ is the orthogonal projection of $r_i - q$ onto $H\in Gr(d, D)$, and $ \theta_2(t) $ is a fast decaying radial weight function consistent across scales (i.e.,$ \theta_2(x_i - q) = \theta_h(\norm{x_i - q}) $ and $ \theta_h(th) = \Phi(t) $ ).
That is,
\begin{equation}\label{eq:argmin_step2}
	\pi^*(x ~|~ r) = \argmin_{\pi\in\Pi_{k-1}^d} J_2(\pi(x) ~|~ q(r), H(r))
.\end{equation}
We define the Manifold-MLS projection as
\begin{equation}
\PP_k^h(r) = \pi^*(0 ~|~ r)
.\label{eq:PPdef}\end{equation}

Using this procedure we can now define the approximating manifold as
\begin{equation}\label{eq:defMh}
\Mh = \{\PP_k^h(p) ~|~ \forall p\in \MM\}
,\end{equation}
where $ h = h_{\MM, R} $ is the fill-distance of the sample-set (see Fig. \ref{fig:Mh}).

\begin{figure}[ht]
	\centering
	\includegraphics[scale=0.6]{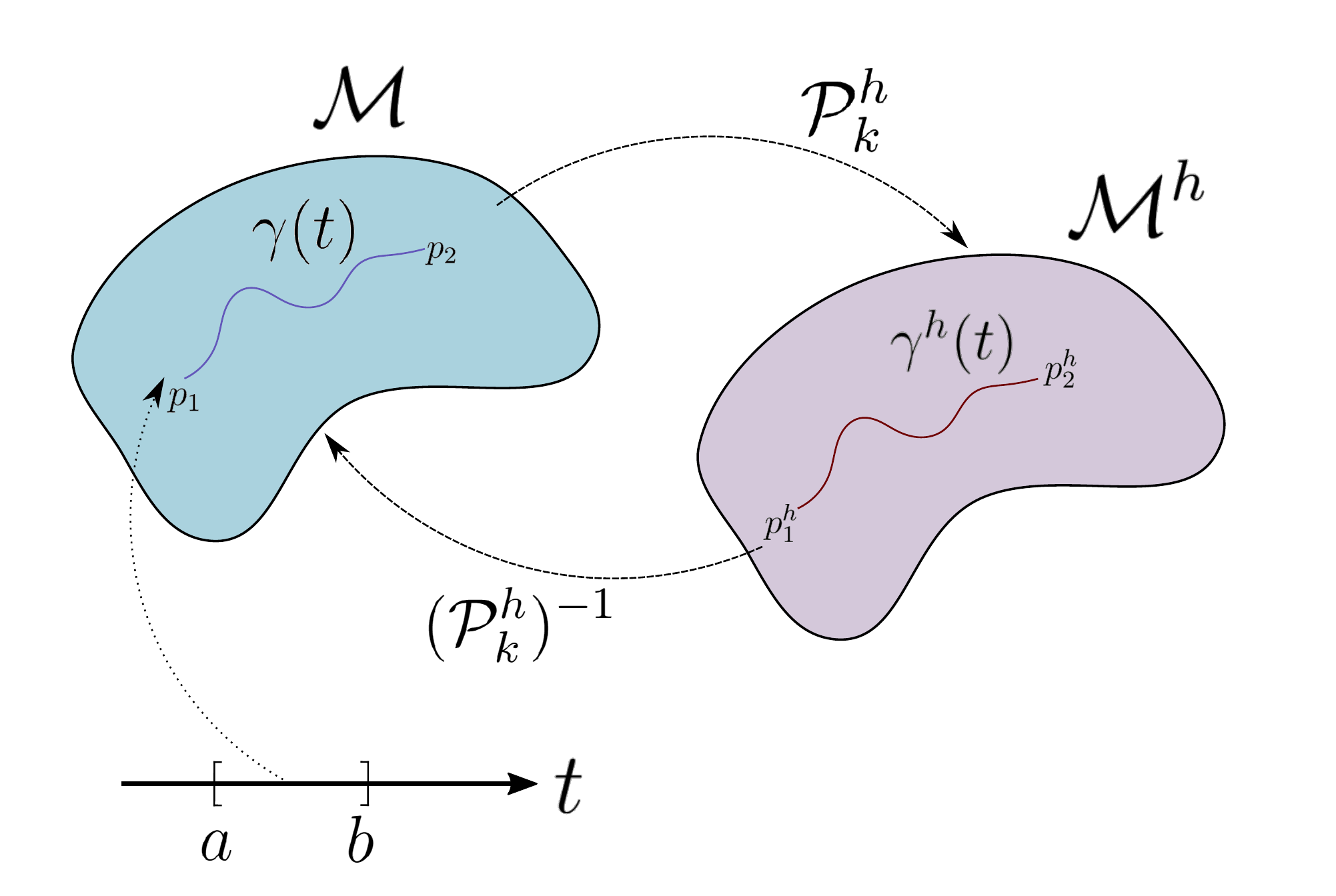}
	\caption{An illustration of a curve $ \gamma^h(t) $ on $ \Mh $ and the corresponding curve $ \gamma_h(t) $ on $ \MM $.}
	\label{fig:Mh}
\end{figure}

As shown in \cite{sober2019manifold}, given that $ \theta_1 , \theta_2 $ are smooth and $h$ small enough (see the Injectivity Conditions of Section \ref{sec:Injectivity}), we achieve that $ \Mh $ is a smooth sub-manifold of $ \RR^D $ and that 
\begin{equation}\label{eq:ConvergenceRate}
	\norm{\Mh - \MM}_{\mathrm{H}} \leq M_0\cdot h^{k}
,\end{equation}
where $ \norm{\cdot}_{\mathrm{H}} $ denotes the Hausdorff norm and $ M $ is some constant.
\begin{remark}
	Note \cite{sober2019manifold} uses the notion of a \hrho sample set rather than a quasi-uniform sample set assumption of Definition \ref{def:quasi-uniform}.
	However, we show in Appendix \ref{sec:Appendix} that these two definitions are the same.
	Therefore, throughout the paper we refer only to the quasi-uniform assumption and use the results of \cite{sober2019manifold} as-is.
\end{remark}
\begin{remark}
    We emphasize that, throughout the exposition of the Manifold-MLS, we have denoted by $h$ the fill distance $h_{R, \MM}$.
    From here on, $h$ will be used in this meaning, and whenever we wish to define a fill-distance with respect to a different sample or a different domain we will denote it explicitly.
\end{remark}

\subsection{Main Result}
Computing geodesic distances on clouds of points are done regularly through finding shortest paths on distance graphs (e.g., by means of Dijkstra algorithm).
Given a limited sample of $ \MM $ the computation of shortest paths may be problematic as shown in Fig. \ref{fig:DiscreteShortestPath_Problem}.
Nonetheless, the Manifold-MLS approximant $ \Mh $ can be sampled in any desired resolution, by taking points $ r $ in a neighborhood of $ \MM $ and computing $ \PP^h_k(r) $.
Thus, the geodesic distances of $ \Mh $ can be attained up to machine precision.
The main contribution of the current paper is in showing that the geodesic distances (as well as paths) computed on $ \Mh $ approximate geodesic distances of $ \MM $ up to $ O(h^{k-1}) $. 
\begin{theorem}\label{thm:MainResult}
	Let the Sampling Assumptions of Section \ref{sec:CleanSampling} hold and let the Injectivity Conditions of Section \ref{sec:Injectivity} hold. 
	Let $\theta_2(x) = \theta_h(\norm{x - q})$ of \eqref{eq:J2def} be such that $\lim_{t\to 0}\theta_h(t) = \infty$ \emph{(}i.e., the Manifold-MLS is interpolatory\emph{)}.
	Let $ \rho_{\MM}, \rho_{\Mh} $ denote the natural distances on $ \MM $ and $ \Mh $ correspondingly, and let $ p_1^h = \PP^h_k(p_1) $ and $ p_2^h = \PP^h_k(p_2) $.
	Then, there exists $h_0$ such that for all $ h \leq h_0 $ we have
	\begin{equation*}
	\rho_{\Mh}(p_1^h,p_2^h) = \rho_{\MM}(p_1,p_2)(1 + O(h^{k-1}))
	\end{equation*}
\end{theorem}
The complete proof of this theorem can be found in Section \ref{sec:proof} below.

\section{Proof of Theorem \ref{thm:MainResult}}\label{sec:proof}

\subsection{The Approximate Isometry}
As explained in Section \ref{sec:MMLS} the first step of the Manifold-MLS is aimed at finding a local coordinate system $ (q, H) $ by which we can describe the manifold locally as the image of a function $ \varphi:H\simeq\RR^d\to\RR^D $.
In other words, $ \varphi(x) $ can be viewed as a local parametrization of the manifold. 
Then, in the second step of the Manifold-MLS we perform locally weighted least-squares polynomial estimation of this function $ \varphi $.
Using this local polynomial estimation of the local parametrization, we show that the metric tensor of $\Mh$ resembles the one of $\MM$. 

\paragraph{The one-to-one relationship:}
Let $ p^h $ be some point in $ \Mh $.
Then, by its construction, we know that there exists $ p_h\in\MM $ such that $ \PP_k^h(p_h) = p^h $ along with the corresponding coordinate domain $ (q(p_h), H(p_h)) $.
In Lemma \ref{lem:1to1} below, we show the existence of $ \varphi_{p_h}:H(p_h)\to\MM $ such that $ \varphi_{p_h}(0) = p_h $.
Furthermore, we also show that if the Injectivity Conditions presented in \cite{sober2019manifold} hold, then this is a one-to-one correspondence.
Translating these conditions to our sampling case we get
\subsubsection*{Injectivity Conditions}
\label{sec:Injectivity}
\begin{enumerate}
	\item The functions $\theta_1(t)$ and $\theta_h(t)$  of \eqref{eq:Step1Minimization} and \eqref{eq:argmin_step2} are monotonically decaying and supported on $[0, c_1 h]$, where $c_1>3$.
	\item Suppose that $\theta_1(c_2h)>c_3>0$ and $\theta_h(c_2h)>c_3>0$, for some constant $c_2< c_1$.
	\item Set $\mu = rch(\MM)/2 $ in constraint \ref{init_constraint:search} of \eqref{eq:Step1Minimization}, where $ rch(\MM) $ denotes the manifold's reach (i.e., the injectivity radius; see \cite{federer1959curvature}).
\end{enumerate}

\begin{lemma}\label{lem:1to1}
	Let the Sampling Assumptions of Section \ref{sec:CleanSampling} and the Injectivity Conditions of Section \ref{sec:Injectivity} hold.
	Then, for all $ h\leq h_0 $ and all $ p\in \MM $ we have:
	\begin{enumerate}
		\item There exists a unique $ p^h\in\Mh $ and a unique coordinate domain $ (q(p), H(p))\in \RR^D\times Gr(d,D) $.
		\item There exist neighborhoods $ W_p\subset\MM, U\subset H(p) $ of $  p, q(p) $ correspondingly such that
		\begin{align*}
		&\varphi_{p}:U\to \MM\subset\RR^D, \quad \varphi^h_{p}:U\to \Mh\subset\RR^D\\
		\end{align*}
		and 
		\begin{align*}
		&\varphi_{p}[U] = W_p\subset\MM, \quad \varphi^h_{p}[U] = W^h\subset\Mh\\
		.\end{align*}
		\item Furthermore, 
		\begin{align*}
		&\varphi_{p}(0) = p, \quad
		\varphi^h_{p}(0) = \pi^*(0 ~|~ p) = p^h
		\end{align*}
	\end{enumerate}
\end{lemma}
\begin{proof}
	Under the Sampling Assumptions and the Injectivity Conditions, the conditions for Lemmas 4.13 and 4.18 of \cite{sober2019manifold} hold.
	Therefore, we know that both $ (q(p), H(p)):\MM\to \RR^D\times Gr(d,D) $ and $ \PP_k^h:\MM\to\Mh $ are injective. 
	This proves item 1 of the lemma.
	From Lemma 4.4. of \cite{sober2019manifold} we know that $ H(p) $ converges to $ T_{p}\MM $ as $ h\to 0 $.
	Thus, there exists a small enough $ h $ such that locally $ \MM $ is the image of the following function:
	\begin{align}
	\begin{split}
	\varphi_{p}:H(p)&\to W_p\subset\MM\\
	x&\mapsto \underbrace{q(p)}_{\text{origin}} + \underbrace{x}_{\text{``tangential"}} + \underbrace{\phi_{p}(x)}_{\text{``normal"}}
	\end{split},\label{eq:phi_p}\end{align}
	where $ \phi_{p}(x)\in H(p)^\perp $.
	The function $ \varphi_{p} $ comprises three parts: a shift to the local origin $ q(p) $, a movement along $ H(p) $ (this is what we mean by ``tangential") and a movement along $ H(p)^\perp $ (normal direction to $ H(p) $); see Fig. \ref{fig:SkewedH} for an illustration.
	
	In addition, since $ \Mh $ is constructed by a local polynomial above $ H $, which is smoothly varying, we can create $ \varphi_{p}^h $, a parametrization from $ H $ to $ \Mh $ with similar components.
	
	\begin{figure}[ht]
		\centering
		\includegraphics[width=\linewidth]{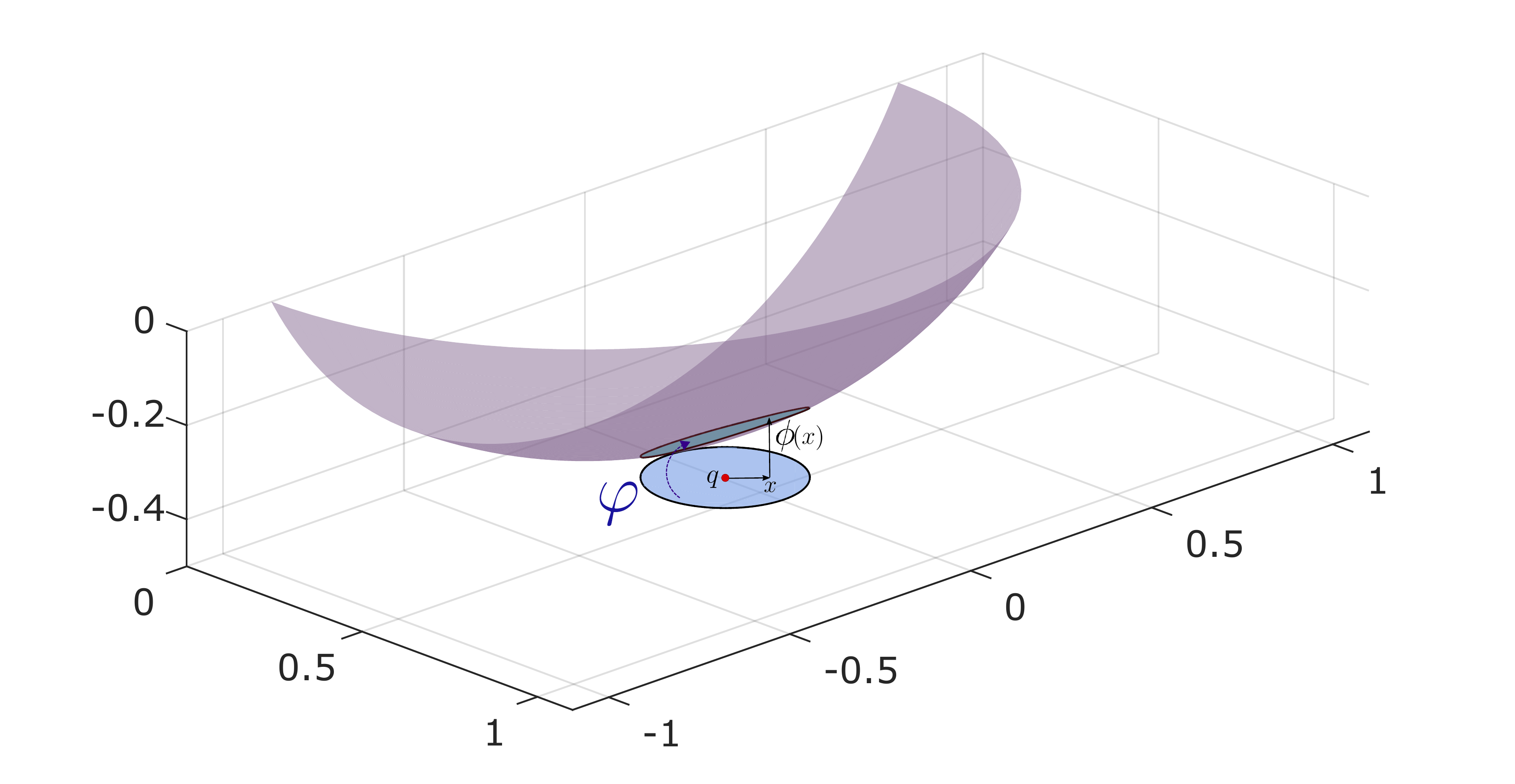}
		\caption{Illustration of $ \varphi_{p} $: The $ xy $-plane is $ H(p) $; The local origin $ q(p) $, which is mapped by $ \phi_{p} $ to $ p $ is marked by the red dot; the vector $ x $ represents a ``tangential" movement; and $ \phi_{p}(x) $ is a normal movement. }
		\label{fig:SkewedH}
	\end{figure}

	From Constraint \ref{init_constraint:perp} of \eqref{eq:Step1Minimization} we know that 
	\[
		p - q(p) \perp H(p)
	.\]
	Thus, the tangential component is null and we have
	\[
	\varphi_{p}(0) = p
	.\]
	Furthermore from \eqref{eq:PPdef} we know that
	\[
	\pi^*(0 ~|~ p) = p^h
	.\]
	Thus, we achieve Results 2 and 3 as required. 
\end{proof}
In other words, under the Injectivity Conditions, $ \pi^*(x ~|~ p) $ is a local weighted polynomial least-squares approximation to $ \varphi_{p} $.
Thus, from Lemma \ref{lem:MLS_Deriv} we get that for $ x\in B_{sh}(0)\subset H(p) $
\begin{equation}
\norm{\varphi_{p}^h(x) - \varphi_{p}(x)}_\infty \leq c h^{k}\norm{\varphi_{p}}_{\C^k(B_{sh}(0))}
,\end{equation}
where 
\[
\norm{\varphi_{p_h}}_{\C^k(B_{sh}(0))} = \max_{\begin{subarray}{c} \abs{\alpha} = k \\ x\in B_{sh}(0)	\end{subarray}} \abs{\partial^\alpha \varphi_{p_h}(x)} 
.\] 
Furthermore, as $ \MM $ is compact and the dependency of $\varphi_{p}$ on $p$ is smooth (resulting from the fact that $H(p)$ vary smoothly -- see Theorem 4.11 in \cite{sober2019manifold}) we can bound $ \norm{\varphi_{p}}_{\C^k(B_{sh}(0))} $ independently of $ p\in \MM $ by a constant and get
\begin{equation}\label{eq:pi_est}
\norm{\varphi_{p_h}^h(x) - \varphi_{p_h}(x)}_\infty \leq C_0 h^{k}
,\end{equation}
for all $ x\in B_{sh}(0) $.
Note, that this implies that
\[
\norm{p^h - p}_\infty = \norm{\varphi_{p}^h(0) - \varphi_{p}(0)}_\infty  \leq C h^{k}
.\]
In Lemma \ref{lem:MMLS_Deriv_Apprxoximation} below, we extend this result to the derivatives as well.
Namely, we show that for $ x\in B_{sh}(0) $
\[
\norm{\partial_{v}\varphi_{p_h}(x) - \partial_{v}\varphi_{p}^h(x)}_\infty \leq C_1 h^{k-1}
.\]

\begin{lemma}\label{lem:MMLS_Deriv_Apprxoximation}
    	Let the Sampling Assumptions of Section \ref{sec:CleanSampling} and the Injectivity Conditions of Section \ref{sec:Injectivity} hold, and let $\theta_2(x) = \theta_h(\norm{x - q})$ of \eqref{eq:J2def} be such that $\lim_{t\to 0}\theta_h(t) = \infty$ \emph{(}i.e., $ \Mh $ interpolates $ \MM $ at the samples $ r_i $\emph{)}.
	Then, there exists a constant $h_0$ such that for all $ h\leq h_0 $, all $ p\in \MM $, any direction $v\in\RR^d$, and all $ x\in B_{sh}(0)\subset H(p) $
\begin{align}\label{eq:pi_derivative_est}
\begin{split}
	\norm{\partial_{v}\varphi_{p}(x) - \partial_{v}\varphi_{p}^h(x)}_\infty &\leq C_1 h^{k-1}, \\
	\norm{\partial_{v}\pi^*(x ~|~ p) - \partial_{v}\varphi_{p}(x)}_\infty &\leq C_2 h^{k-1}\\
	\norm{\partial_{v}\pi^*(x ~|~ p) - \partial_{v}\varphi_{p}^h(x)}_\infty &\leq C_3 h^{k-1}
,\end{split}
\end{align}
where $C_1, C_2, C_3$ are constants independent of $v$ or $p$, and $\partial_{v}\pi^*(\hat x ~|~ p) = \at{\dxj{v}}{x = \hat x}\pi^*(x ~|~ p)$ denotes the directional derivative of the polynomial in the direction $v$ with respect to the first input evaluated at $\hat x \in B_{sh}(0)$; i.e., taking $p$ to be constant.
\end{lemma}

\begin{proof}
Let us fix an arbitrary point $p\in \MM$. 
According to Lemma \ref{lem:1to1} we have the local coordinate system $(q, H) = (q(p), H(p))$ and the local functions $\varphi_{p}:H\to\MM$ and $\varphi_{p}^h:H\to\Mh$, defined on some open ball around the origin $B_0(s)\subset H\simeq \RR^d$.
Furthermore, we define the sample sites on $H$ by means of
\[
X = \{x_i = P_H(r_i - q) ~|~ \norm{r_i - p} \leq \mu\}_{i=1}^n\cap B_0(sh)
,\]
and by Lemma \ref{lem:Projected_h} we know that $h_{X, B_0(sh)} \leq c \cdot h$ for some constant $c$.

Under these settings we can now define $s_{\varphi, X}(x), s_{\varphi^h, X}(x)$, the MLS approximants of $\varphi_{p}$ and $\varphi_{p}^h$ respectively. 
As we know that $\lim_{t\to 0}\theta_h(t) = \infty$, we get that the Manifold-MLS is interpolatory (see Section 3.2 in \cite{sober2017approximation}) and so
\[
\varphi_{p}(x_i) = r_i = \varphi_{p}^h
.\]
Therefore, we get that 
\[
s_{\varphi, X}(x) = s_{\varphi^h, X}(x) = \sum_{x_i\in X} a_i^*(x) r_i
.\]
Using Lemma \ref{lem:MLS_Deriv}, we get that for $ x\in B_{sh}(0) $
\[
\norm{\partial_{v}s_{\varphi, X}(x) - \partial_{v}\varphi_{p}(x)}_\infty \leq \tilde c_1\norm{\varphi_{p}}_{\C^k(B_{sh}(0))} h^{k-1}
,\]
and
\[
\norm{\partial_{v}s_{\varphi^h, X}(x) - \partial_{v}\varphi_{p}^h(x)}_\infty = \norm{\partial_{v}s_{\varphi, X}(x) - \partial_{v}\varphi_{p}^h(x)}_\infty \leq \tilde c_2 \norm{\varphi_{p}^h}_{\C^k(B_{sh}(0))} h^{k-1}
,\]
where 
\[
 \norm{f}_{\C^k(B_{sh}(0))} \defeq \max_{1 \leq \ell \leq D}\abs{f^\ell}_{\C^k(B_{sh}(0))} = \max_{1 \leq \ell \leq D} \max_{\begin{subarray}{c}
 	\abs{\alpha} = k \\ x\in B_{sh}(0)
 	\end{subarray}} \abs{ \partial^\alpha f^\ell(x)}
.\]
We know from Theorem 4.11 of \cite{sober2019manifold} that $H(p)$ depend smoothly on $p$ we get that both $\varphi_{p}(x)$ and $\varphi_{p}^h(x)$ vary smoothly with $p$.
Combining this with the fact that $\MM$ is compact we can give a global constant bound for both $\norm{\varphi_{p}}_{\C^k(B_{sh}(0))}$ and $\norm{\varphi_{p}^h}_{\C^k(B_{sh}(0))}$.
Thus we obtain
\begin{align*}
\norm{\partial_{v}s_{\varphi, X}(x) - \partial_{v}\varphi_{p}(x)}_\infty \leq c_1 h^{k-1} \\
\norm{\partial_{v}s_{\varphi, X}(x) - \partial_{v}\varphi_{p}^h(x)}_\infty \leq c_2  h^{k-1}
\end{align*}
Then, by the triangle inequality we achieve
\begin{align*}
\norm{\partial_{v}\varphi_{p}^h(x) -  \partial_{v}\varphi_{p}^h(x)}_\infty 
&\leq 
\norm{\partial_{v}\varphi_{p}^h(x) -\partial_{v}s_{\varphi, X}(x) } +
\norm{\partial_{v}s_{\varphi, X}(x) - \partial_{v}\varphi_{p}(x)} \\
&\leq c_1 h^{k-1} + c_2 h^{k-1} = C_1 h^{k-1}
.\end{align*}

Finally, we note that the polynomial $\pi^*(x ~|~ p)$ which minimizes \eqref{eq:argmin_step2} coincides with the polynomial $\pi^*(x ~|~ 0)$ minimizing \eqref{eq:MLS_basic_Minimization} with respect to the domain $H(p)$.
As described in the alternative minimization problem of \eqref{eq:MLS_Basic_Quadratic_Form}, it reproduces polynomials in $\Pi_{k-1}$, and more specifically, the Taylor polynomial of degree $k-1$.
Therefore, we get for $ x\in B_{sh}(0) $
\begin{align*}
\norm{\partial_{v}\pi^*(x ~|~ p) - \partial_{v}\varphi_{p}(x)}_\infty \leq C_2 h^{k-1}\\
\norm{\partial_{v}\pi^*(x ~|~ p) - \partial_{v}\varphi_{p}^h(x)}_\infty \leq C_3 h^{k-1}
\end{align*}
as required.
\end{proof}

An immediate consequence of Lemma \ref{lem:MMLS_Deriv_Apprxoximation} is that the differential of $\varphi_{p}^h$ approximates the differential of $\varphi_{p}$
\begin{lemma}\label{lem:phi_diff_approx}
    Let the conditions of Lemma \ref{lem:MMLS_Deriv_Apprxoximation} hold, and let $\DD_0\varphi_{p}, \DD_0\varphi^h_{p}$ denote the differentials of $\varphi_{p}, \varphi^h_{p}$ at $0$ respectively.
    Then,
    \[
    \DD_0\varphi_{p} = \DD_0 \varphi_{p}^h + X
    ,\]
    where $X$ is a linear operator with operator norm
    \[
    \norm{X}_{op}  \leq C h^{k-1}
    ,\]
    for some constant $C$.
\end{lemma}
\begin{proof}
    Let us define 
    \[
    X = \DD_0\varphi_{p} - \DD_0 \varphi_{p}^h
    ,\]
    which is linear as a subtraction of two linear operators.
    Recall that vectors $v\in T_0H(p) = H(p)$ naturally correspond to directional derivatives; explicitly,
    \begin{align*}
        \DD_0\varphi_{p}\cdot v &= \partial_v \varphi_{p}(0) \\ 
        \DD_0\varphi^h_{p}\cdot v &= \partial_v \varphi^h_{p}(0)
    .\end{align*}
    From Lemma \ref{lem:MMLS_Deriv_Apprxoximation} we know that for any such $v$ we get
    \[
    \norm{X\cdot v} = \norm{\DD_0\varphi_{p}\cdot v  - \DD_0\varphi^h_{p}\cdot v} = \norm{\partial_v \varphi_{p}(0) - \partial_v \varphi^h_{p}(0)} \leq C_1 h^{k-1}
    ,\]
    since the Euclidean norm respects the inequality $\norm{x} \leq \norm{x}_\infty$ (essentially, this claim will be true in any given norm due to the equivalence of norms in finite dimensions).
    Thus,
    \[
    \norm{X}_{op} = \max_{\norm{v}=1}\norm{X \cdot v} \leq C h^{k-1}
    \]
\end{proof}

\begin{corollary}\label{cor:LocalIsometry}
	Define the local map $ \mu_p:W_p\to W^h $ by $ \varphi_{p}^h \circ \varphi_{p}^{-1} $.
	Then, its differential 
	\[ 
	\DD_p \mu_p = Id + Y 
	,\]
	where 
	\[
	\norm{Y}_{op} \leq C h^{k-1}
	,\]
	for some constant $C$.
\end{corollary}
\begin{proof}
	Let us write $\mu_p$ as the following composition
	\[
	\mu_p = \varphi_p^h \circ \varphi_p^{-1}:W_p \to W^h
	,\]
	where $\varphi_p:U\to W_p\subset \MM$ and $\varphi_p^h:U\to W^h\subset\Mh$ as defined in Lemma \ref{lem:1to1}.
	Then, by the chain rule we get
	\[
	\DD_p\mu_p = \DD_0 \varphi_p^h \circ \DD_p\varphi_p^{-1}
	.\]
	Furthermore, from the fact that $Id =\DD_p(Id)$ we know
	\begin{align*}
	Id = \DD_p(\varphi_p\circ\varphi_p^{-1}) = \DD_0\varphi_p\circ\DD_p\varphi_p^{-1}
	,\end{align*}
	and so we get
	\[
	\DD_p\varphi_p^{-1} = (\DD_0 \varphi_p)^{-1}
	.\]
	In addition, we know from Lemma \ref{lem:phi_diff_approx} that
	\[
	\DD_0 \varphi_p = \DD_0 \varphi_p^h + X
	,\]
	where $\norm{X}_{op} \leq c_1 h^{k-1}$.
	Applying $(\DD_0\varphi_p)^{-1}$ to both sides of the equation results with 
	\begin{align*}
	Id = \DD_0\varphi_p^h \circ (\DD_0\varphi_p)^{-1} &+ X \circ (\DD_0\varphi_p)^{-1} \\
	Id -  \DD_p\PP_k^h &= X \circ (\DD_0\varphi_p)^{-1}
	.\end{align*}
	Denoting $X \circ (\DD_0\varphi_p)^{-1}$ by $Y$ we get the desired expression, and
	\[
	\norm{Y}_{op} =  \norm{X \circ (\DD_0\varphi_p)^{-1}}_{op} \leq 
	\norm{X}_{op} \norm{ (\DD_0\varphi_p)^{-1}}_{op} = \norm{X}_{op} \frac{1}{\norm{ \DD_0\varphi_p}_{op}}
	.\]
	Since the manifold is differentiable and compact, and $H$ approximates the tangent, we can bound $\norm{ \DD_0\varphi_p}_{op}$ from below by $0<1/c_2$ a constant independent of $p$.
	Therefore,
	\[
	\norm{Y}_{op} \leq c_1 c_2 h^{k-1} = C h^{k-1}
	,\]
	as requested.
\end{proof}

\begin{lemma}\label{lem:Projected_h}
Let the Sampling Assumptions of Section \ref{sec:CleanSampling} and the Injectivity Conditions of Section \ref{sec:Injectivity} hold.
Let $\theta_1, \theta_2$ of \eqref{eq:J1def} and \eqref{eq:J2def} be with supports of size $c_1 h$ and $s h$ respectively and let $s \leq c_1$. 
Fix an arbitrary point $p\in \MM$ and denote the corresponding coordinate system resulting from \eqref{eq:Step1Minimization} by $(q, H) = (q(p), H(p))$, and let 
$$X = \{x_i = P_H(r_i - q) ~|~ \norm{r_i - p} \leq \mu\}_{i=1}^n\cap B_0(sh),$$
where $s$ is some constant and $\mu$ is smaller than the injectivity radius (i.e., the manifold's reach).
Then, for $h_{R, \MM}$ sufficiently small we get
\[
h_{X, B_0(sh)} \leq c \cdot h_{R, \MM}
,\]
where $c$ is some constant independent of $p$.
\end{lemma}
\begin{proof}
The proof outline is as follows
\begin{enumerate}
    \item We first look at the sample set projected onto the tangent $T_{p}\MM$,
    \[
    \tilde X = \{\tilde x_i = P_{T_{p}\MM}(r_i - p) ~|~ \norm{r_i - p}\leq \mu\}_{i=1}^n\cap B_0(c_1 h)
    \]
    and its fill-distance 
    \[
    h_{\tilde X, B_0(c_1 h)} \leq h_{R, \MM}
    .\]
    \item Then by Lemma \ref{lem:h-delta-is-rho} from Appendix A we achieve that there exists a constant $\rho$ such that the number of samples in $\tilde X$ is bounded by $c_1 \cdot \rho^d$.
    \item As a result, we achieve that $J_1(p, T_{p}\MM) = O(h^4)$.
    \item Since $ (q(p), H(p))$ are the minimizers of \eqref{eq:Step1Minimization} and since $(p, T_{p}\MM)$ are in the search space of the minimization problem of \eqref{eq:Step1Minimization}, we get that 
    \[
    J_1(q(p), H(p)) \leq J_1(p, T_{p}\MM) = O(h^4)
    \]
    \item Since the number of samples in $X$ is bounded by $c_1\cdot \rho^d$ as well we achieve that
    \[
    d(r_i - q, H) = O(h^2)
    \]
    \item Thus, for small enough $h_{R, \MM}$ we get that 
    \[
    h_{X, B_0(c_1 h)} \leq c\cdot h_{R, \MM}
    .\]
\end{enumerate}

Looking at the projected sample
    \[
    \tilde X = \{\tilde x_i = P_{T_{p}\MM}(r_i - p) ~|~ \norm{r_i - p}\leq \mu\}_{i=1}^n\cap B_0(c_1 h)
    ,\]
and measuring the fill-distance of that set we get
\begin{align*}
h_{\tilde X, B_0(c_1 h)} &= \sup_{\tilde x\in B_0(c_1 h)} \min_{\tilde x_i\in \tilde X}\norm{\tilde x - \tilde x_i} \\
&= \sup_{ r\in \phi[B_0(c_1 h)]} \min_{r_i\in R\cap \phi[B_0(c_1 h)]}\norm{P_{T_{p}\MM}(r - p) - P_{T_{p}\MM}(r_i - p)} 
,\end{align*}
where $r = \phi(\tilde x)$ is the unique point which satisfies $P_{T_{p}\MM}(r - p) = \tilde x$.
Thus, 
\begin{align*}
h_{\tilde X, B_0(c_1 h)} 
&\leq \sup_{ r\in \phi[B_0(c_1 h)]} \min_{r_i\in R\cap \phi[B_0(c_1 h)]}\norm{r - r_i}\\
&\leq h_{R, \MM}
.\end{align*}
Furthermore, since the tangent is a linear approximation at $p$ we know that for $r_i\in \phi[B_0(c_1 h)]$
\[
d(r_i - p, T_{p}\MM) = O(h^2)
.\]
Thus, for any $\varepsilon$ there is small enough $h$ such that 
\[
\phi[B_0(c_1 h)]\subset B_{p}((c_1 + \varepsilon)h))
.\]
From the fact that $R$ is a quasi-uniform set and Lemma \ref{lem:h-delta-is-rho} we know that
\[
\# (R\cap B_{p}((c_1 + \varepsilon)h)) \leq (c_1 + \varepsilon) \cdot \rho^d
\]
for some constant $\rho$.
Thus, we achieve that 
\[
\#\tilde X\leq (c_1 + \varepsilon) \cdot \rho^d
\]
as well.

Now we show that $J_1(q(p), H(p) ~|~ p) = O(h^4)$ by showing that 
\[
J_1(q(p), H(p) ~|~ p) \leq J_1(p, T_{p}\MM ~|~ p) = O(h^4)
.\] 
The coordinate system $(q(p), H(p))$ is the one minimizing the following energy
\[
J_1(q, H ~|~ r) = \sum_{i=1}^{n} d(r_i-q , H)^2 \theta_1(\| r_i - q\|) 
\]
under the constraints of \eqref{eq:Step1Minimization}:
\begin{enumerate}
\item $r-q \perp H$  
\item $q\in B_{\mu}(r)$ 
\item $\#\left(R\cap B_{ h}(q)\right) \neq 0$
\end{enumerate}
We note that a possible choice of coordinates that respect these three constraints is given by taking $q = p$ and $H = T_{p}\MM$. 
Thus, we achieve
\[
J_1(q(p), H(p) ~|~ p)  \leq J_1(p, T_{p}\MM ~|~ p) 
\]
, and
\[
J_1(p, T_{p}\MM ~|~ p) = \sum_{i=1}^{n} d(r_i-p , T_{p}\MM)^2 \theta_1(\| r_i - p\|) 
.\]
Since $T_{p}\MM$ is the linear approximation to $\MM$ at $p$, we know that for $r_i$ in the support of $\theta_1$ (i.e., $\norm{r_i - p} \leq c_1 h$)
\[
\norm{P_{T_{p}\MM}(r_i - p) } \leq c_1 h_{R, \MM}
,\]
and thus,
$$d(r_i - p,T_{p}\MM) \leq c_2 h_{R, \MM}^2.$$
Furthermore, since $\#\tilde X \leq (c_1 + \varepsilon)\rho^d$ we achieve that
\[
J_1(p, T_{p}\MM ~|~ p) = O(h^4)
\]
and, as a result
\[
J_1(q(p), H(p) ~|~ p) = O(h^4)
\]
as well.
From Lemma 4.4 of \cite{sober2019manifold} we know that under the Injectivity Conditions as $h\to 0$ we achieve that $q(p)$ approaches $p$, and if we denote $\epsilon = \norm{q(p) - p}$, then
\[
\norm{P_{H(p)} - P_{T_{p}\MM}}_{op} \leq O(h + \epsilon^2)
.\]
Thus, for sufficiently small $h$ we get that
\[
\#X \leq (s + \tilde \varepsilon)\cdot \rho^d
\]
and we get that 
\[
d(r_i - q(p), H(p)) = O(h^2)
.\]
Finally,
\begin{align*}
    h_{X, B_0(s h)} &= \sup_{x\in B_0(s h)}\min_{x_i\in X} \norm{x - x_i}\\
    &=\sup_{r\in \varphi_{p}[B_0(s h)]}\min_{r_i\in R\cap \varphi[B_0(s h)]} \norm{P_H(r - q) - P_H(r_i - q)}\\
    &\leq
    \sup_{r\in \varphi_{p}[B_0(s h)]}\min_{r_i\in R\cap \varphi[B_0(s h)]} \norm{r - r_i} + O(h^2)
.\end{align*}
Thus, for sufficiently small $h$ we obtain
\[
h_{X, B_0(s h)} \leq c \cdot h_{R, \MM}
.\]
\end{proof}

\subsection{Geodesic Paths on $ \Mh $ Converge to Geodesic Paths in $ \MM $}
Let us fix points $ p_1, p_2\in \MM $ and denote a path between them by $ \gamma(t):[a,b]\to\MM $explicitly, $ \gamma(a) = p_1, \gamma(b) = p_2 $.
Then, the length of the curve is defined as
\[
L_g(\gamma(t)) = \int_a^b\norm{\dot \gamma(t)}_g dt
,\]
and the geodesic distance can be defined by the \emph{piece-wise smooth} curve $\gamma(t):[a,b]\to\MM$ of minimial length such that $\gamma(a) = p_1, \gamma(b) = p_2$.
Namely,
\begin{equation}\label{eq:geodesic}
\rho_\MM(p_1, p_2) = \min_{\gamma(t)}L_g(\gamma(t)) = 
\min_{\gamma(t)}\int_a^b\norm{\dot \gamma(t)}_g dt
.\end{equation}
Notice, that one can always choose the arc-length parameter defined by 
\[
s(t) = \int_a^{t}\norm{\dot \gamma(t)}_g dt
\]
and then get
\[
L_g(\gamma(t)) = L_g(\gamma(s(t))) = \int_a^b ds
.\]
Thus, when looking for a geodesic curve (i.e., one that minimizes the length) we can limit the discussion to curves with arc-length parameter; that is $\norm{\dot \gamma(t)}_g = 1$, up to the zero measurable set where $ \gamma $ is not smooth.

\paragraph{A (Simplified) Local Case:}
Let $\gamma^*(t):[a, b]\to\MM$ denote a geodesic curve such that $\gamma^*(a) = p_1, \gamma^*(b) = p_2$ and $\norm{\dot\gamma^*(t)}_g = 1$.
That is,
\[
\rho_\MM(p_1, p_2) = L_g(\gamma^*(t))
.\]
If $ p_1, p_2$ belong to an open neighborhood $W_{p_1}\subset \MM $ such that 
\[
\varphi_{p_1}:U \to W_{p_1}
\]
and $ U\subset B_{sh}(0)\subset H(p_1) $, then we can look at the map $ \mu_{p_1}:W_{p_1}\to W_{p_1}^h $ defined by
\[
\mu = \varphi_{p_1}^h \circ \varphi_{p_1}^{-1}
\]
and from Corollary \ref{cor:LocalIsometry} we know that its differential
\[
\DD_{p_1}\mu_{p_1} = Id + Y 
,\]
where $ \norm{Y}_{op} \leq c h^{k-1} $.

In this case, we could look at the pushed forward curve
\[\gamma^h(t)) = \mu_{p_1}(\gamma^*(t)).\]
By the chain rule, we get
\[
\dot \gamma^h(t) = \at{\partial_t}{t =  t} (\mu_{p_1} \circ \gamma^*(t)) = 
\DD_{p_1}\mu_{p_1} \cdot \partial_t\gamma^*(t) = 
\DD_{p_1}\mu_{p_1} \cdot \dot\gamma^*( t)
.\]
Using Corollary \ref{cor:LocalIsometry} we achieve that
\[
\dot \gamma^h( t) = \dot \gamma^*( t) + Y \cdot \dot \gamma^*( t)
,\]
and thus,
\begin{align}\label{eq:BoundGammaDot}
\begin{split}
\norm{\dot \gamma^h(t)}_{g^h} &= 
\norm{\dot \gamma^*(t) + Y \cdot \dot \gamma^*(t)}_{g}\\ &\leq 
\norm{\dot \gamma^*(t)}_{g} + \norm{Y \cdot \dot \gamma^*(t)}_{g} \\
&\leq 
\norm{\dot \gamma^*(t)}_{g} + \norm{Y}_{op}\norm{ \dot \gamma^*(t)}_{g} \\
&= 1 + \norm{Y}_{op}\\
&\leq 1 + c h^{k-1}
\end{split}
.\end{align}
Since the manifold $\MM$ is closed (i.e., $ \rho_\MM(p_1, p_2)$ can be bounded globally), the length of $\gamma^h(t)$ can now be bounded by
\begin{equation}\label{eq:local_geodesic_bound}
L_{g^h}(\gamma^h(t)) \leq \int_a^b 1 + c h^{k-1} dt = L_g(\gamma^*) + c h^{k-1} L_g(\gamma^*)
.\end{equation}
Assuming the lengths of curves are bounded we get
\[
L_{g^h}(\gamma^h(t)) \leq = \rho_\MM(p_1, p_2)  + c_1 h^{k-1}  
.\]
From the fact that the geodesic curve has minimal length we get
\[
\rho_{\Mh}(p_1^h, p_2^h) \leq \rho_\MM(p_1, p_2) + c_1 h^{k-1}
,\]
where $ \tilde p_1^h = \mu_{p_1}(p_1) $ and $ \tilde p_2^h = \mu_{p_1}(p_2) $.
Unfortunately, this analysis does not solve the issue.
Explicitly, even though
\[
\mu_{p_1}(p_1) = \tilde p_1^h = p_1^h = \PP_k^h(p_1)
,\]
it is apparent that
\[
\mu_{p_1}(p_2) = \tilde p_2^h \neq p_2^h = \PP_k^h(p_2)
\]
(see Figure \ref{fig:zj}).
Therefore, we need to bridge this gap between $ p_2^h $ and $ \tilde p_2^h $ and make sure that it does not hamper the $ O(h^{k-1}) $ convergence rates.
By doing this, we will achieve the desired bound from the triangle inequality.

\paragraph{The (Real) Global Case:} Let us now remove the assumption that $ p_1, p_2 $ belong to such a neighborhood $W$ that can be mapped by $ \varphi_{p_1} $.
In this case we sample the curve $ \gamma^*(t) $ at the points $ \{z_j = \gamma^*(t_j)\}_{j=1}^{J+1} $ densely enough such that each pair $ z_j, z_{j+1} $ belong to an open neighborhood $W_j\subset B_{sh}(z_j)$ that can be mapped by 
\[
\varphi_{ z_j} : U_j\to W_j 
.\]
Then, we look at the set of points 
\begin{align*}
	z^h_j = \PP_k^h(z_j)\in \Mh, &\quad (j = 1\ldots J+1)\\
	\tilde z^h_{j+1} = \mu_{z_j}(z_{j+1})\in \Mh, &\quad (j=1, \ldots, J)
.\end{align*}
Note, that by definition
\[
\mu_{z_j}(z_j) = z^h_j
.\]
Thus, essentially the set of points we are looking at are
\begin{align*}
z^h_j = \mu_{z_j}(z_j)\in \Mh, &\quad (j = 1\ldots J+1)\\
\tilde z^h_{j+1} = \mu_{z_j}(z_{j+1})\in \Mh, &\quad (j=1, \ldots, J)
.\end{align*}
\begin{figure}
	\centering
	\includegraphics[width=0.8\linewidth]{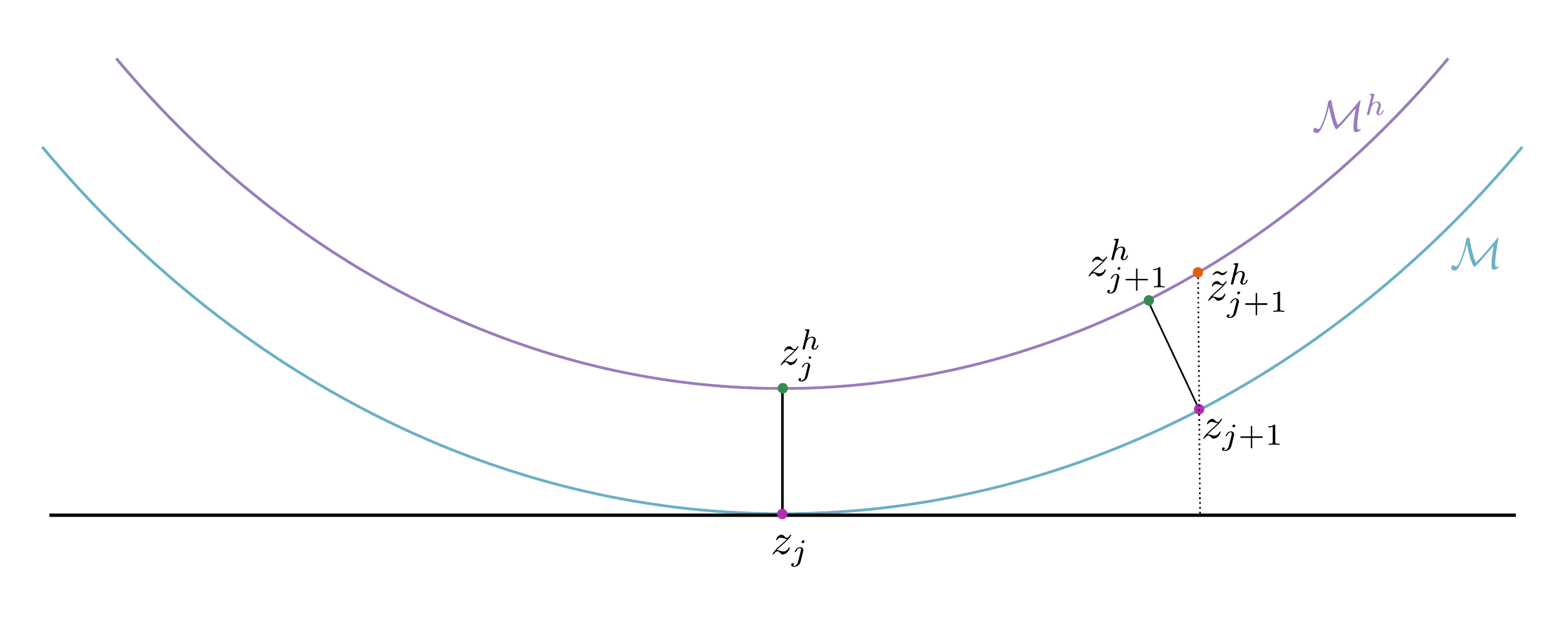}
	\caption{Illustration of the difference between $z_{j+1}^h $ and $ \tilde z_{j+1}^h $ with respect to the coordinate system $ H(z_j) $ (i.e., aligned in a way that $ P_{H(z_j)} (z_j) = P_{H(z_j)}(z^h_j) = 0$). }
	\label{fig:zj}
\end{figure}
Using these sets of points we construct a piece-wise smooth curve $ \gamma^{h,J}(t) $.
We begin by defining the first segment $ \gamma^h_1(t) $ connecting $ z_1^h $ with $ \tilde z^h_2 $ by pushing the curve $ \gamma^*_1 $ (the restriction of $ \gamma^*(t) $ between $z_1$ and $ z_2 $) with $ \mu_{z_1} $.
Explicitly,
\[
\gamma^h_1(t) = \mu_{z_1}\circ\gamma^*_1(t) = \mu_{z_1}\circ\at{\gamma^*(t)}{[z_1, z_2]_\MM}
.\]
Note, that $ \gamma^h_1(t) $ ends at $ \mu_{z_1}(z_2) = \tilde z^h_2 $.
Then, the following segment denoted by $ \tilde \gamma^h_1(t) $ is defined through the geodesic path connecting $ \tilde z^h_2 $ and $ z^h_2 = \PP_k^h(z_2) = \mu_{z_2}(z_2)$.
Subsequently, 
\[
\gamma^h_2(t) = \mu_{z_1}\circ\gamma^*_2(t) = \mu_{z_2}\circ\at{\gamma^*(t)}{[z_2, z_3]_\MM}
,\]
and $ \tilde\gamma^h_2(t) $ is defined through the geodesic path connecting $ \tilde z^h_3 $ and $ z^h_3 = \PP_k^h(z_3) = \mu_{z_3}(z_3)$.
We define the rest of the piece-wise smooth curve in the same manner (see Figure \ref{fig:PiecewiseSmoothCurve}).
Let us denote
\begin{align*}
\ell^h_j = L_{g^h}(\gamma^h_j(t)), &\quad (j=1, \ldots J) \\
\ell^*_j = L_{g}(\gamma^*_j(t)), &\quad (j=1, \ldots J) \\
.\end{align*}
Then, from the considerations portrayed in the Localized Case (i.e., \eqref{eq:BoundGammaDot}) we get that
\[
\ell^h_j \leq \ell^*_j 
.\]
Furthermore, from Lemma \ref{lem:eps_geodesic_bound} we achieve that
\[
L_{g^h}(\tilde\gamma^h_j(t)) \leq c_1 h^{k}
.\]
Therefore,  
\begin{align*}
L_{g^h}(\gamma^{h, J}(t)) 
&= \sum_{j=1}^{J} L_{g^h}(\gamma^h_j(t)) + L_{g^h}(\tilde \gamma^h_j(t)) \\ 
&\leq \sum_{j=1}^{J} \ell^h_j + c_1 h^{k}\\
&\leq \sum_{j=1}^{J} \ell^*_j + c_2 h^{k-1} \ell^*_j + c_1 h^{k} \\ 
&= L_{g}(\gamma^*(t)) +  L_{g}(\gamma^*(t)) \cdot c_2 h^{k-1} + J c_1 h^{k} 
,\end{align*}
where the second inequality comes from \eqref{eq:local_geodesic_bound}.
It is clear that for sufficiently small $h$ we have $ J=O(L_g(\gamma^*(t))/h) $, and so we obtain
\[
L_{g^h}(\gamma^{h, J}(t)) \leq \rho_\MM(z_1, z_{J+1}) +  \rho_\MM(z_1, z_{J+1}) \cdot c_1 h^{k-1} + c_3 h^{k-1}
.\]
Since $ \MM $ is closed (i.e., compact and bounded) and of bounded reach we know that the length of curves on $ \MM $ are bounded, and we obtain
\[
\rho_{\Mh}(p^h_1, p^h_2) \leq \rho_\MM(p_1, p_2)(1 + C_1 h^{k-1})
,\]
where $ p_1 = z_1, p_2 = z_{J+1} $ and $ p^h_1 = z^h_1, p^h_2 = z^h_{J+1} $.

By symmetry, we can show
\[
\rho_\MM(p_1, p_2) \leq \rho_{\Mh}(p_1^h, p_2^h)(1 + C_2 h^{k-1})
,\]
and the proof of Theorem \ref{thm:MainResult} is concluded.\qed
\begin{figure}
	\centering
	\includegraphics[width=0.4\linewidth]{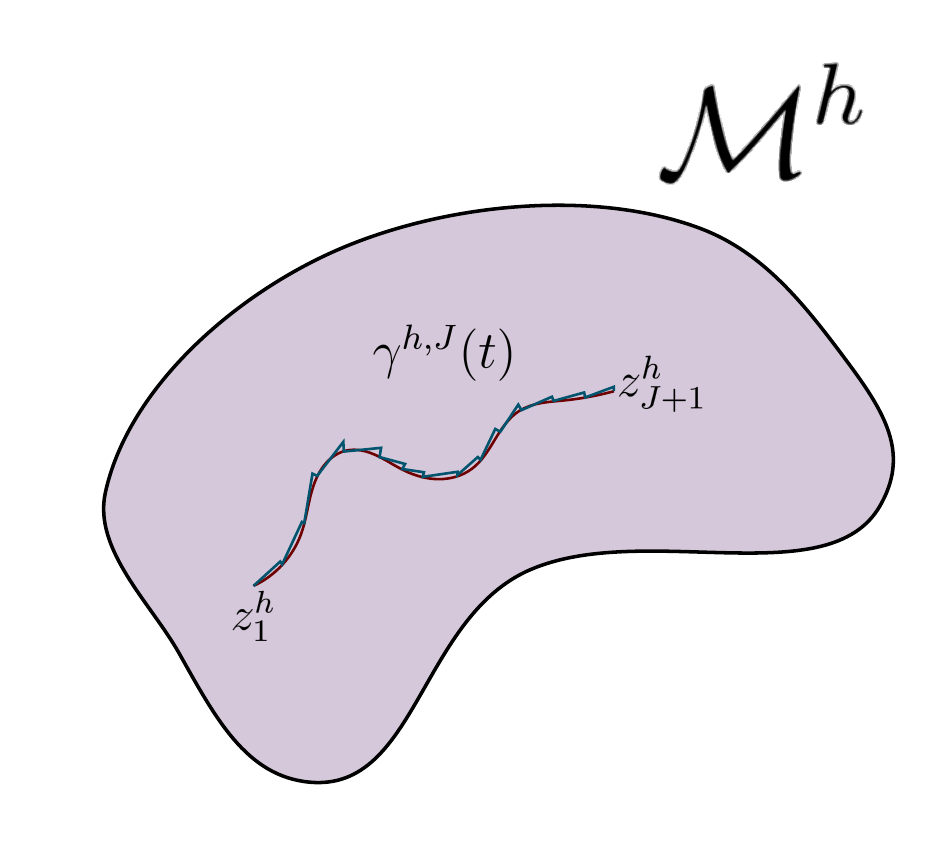}
	\caption{Illustration of $ \gamma^{h, J}(t) $ (the piece-wise smooth curve in blue) drawn near the limit curve $ \gamma^h(t) $ (in red) }
	\label{fig:PiecewiseSmoothCurve}
\end{figure}

\begin{lemma}\label{lem:eps_geodesic_bound}
	Let $\rho_\MM(z_j, z_{j+1}) \leq s h$ then for sufficiently small $ h $
	\[
	\rho_{\Mh}(z^h_{j+1}, \tilde z^h_{j+1}) \leq C h^{k}
	.\]
\end{lemma}
\begin{proof}
	For simplicity we assume that $ q(z_j) = 0 $.
	We first remind that
	\[
	z^h_{j+1} = \PP_k^h(z_{j+1}) 
	,\]
	and thus by \eqref{eq:pi_est}
	\[
	\norm{z^h_{j+1} - z_{j+1}}\leq c_1 h^{k}
	.\]
	Accordingly, we have
	\[
	\norm{P_{H(z_j)}(z_{j+1}) - P_{H(z_j)}(z^h_{j+1})} \leq c_2 h^{k}
	.\]
	On the other hand we have
	\[
	P_{H(z_j)}(z_{j+1}) = P_{H(z_j)}(\tilde z^h_{j+1})
	,\]
	and so
	\[
	\norm{P_{H(z_j)}(\tilde z^h_{j+1}) - P_{H(z_j)}(z^h_{j+1})} \leq c_2 h^{k}
	\]
	as well.
	Now, since $ \varphi^h_{z_j} $ so is $ \phi^h_{z_j}:H(z_j)\to H^\perp_{z_j} $ (as defined in \eqref{eq:phi_p}; see Figure \ref{fig:zj}) is a smooth function, and through taking the zero$ ^\textrm{th} $ order Taylor approximation around $ x_0 = P_{H(z_j)}(z_{j+1}) $ we know that
	\[
	\phi^h_{z_j}(x_0 + \varepsilon) = \phi^h_{z_j}(x_0) + c_3 \varepsilon 
	,\]
	where $ c_3 $ depends on the first derivatives of $ \varphi^h_{z_j} $.
	Since $ \varphi^h_{z_j} $ vary smoothly with the change in $ z_j $ and $ \MM $ is compact, we can bound $ c_3 $ by a uniform constant independent of $ z_j $.
	Therefore, by taking $ \varepsilon =  P_{H(z_j)}(\tilde z^h_{j+1}) - P_{H(z_j)}(z^h_{j+1}) $ we get
	\[
	\norm{\varepsilon} \leq c_2 h^k
	,\]
	and so,  since
	\[
	\norm{\tilde z^h_{j+1} - z^h_{j+1}} =
	\norm{\varphi^h_{z_j}(P_{H(z_j)}(\tilde z^h_{j+1})) - \varphi^h_{z_j}(P_{H(z_j)}(z^h_{j+1}))}
	,\]
	we achieve
	\[
	\norm{\tilde z^h_{j+1} - z^h_{j+1}} = 
	\sqrt{\varepsilon^2 + \left(\phi^h_{z_j}(P_{H(z_j)}(\tilde z^h_{j+1})) - \phi^h_{z_j}(P_{H(z_j)}(z^h_{j+1}))\right)^2} \leq
	c_2 h^k + c_3 h^k = c_4 h^k
	.\]
	Finally, as $ \MM\in \C^2 $ and using a Taylor approximation of $ \MM $ over the tangent at $ z_{j+1}^h $, for sufficiently small $ h $  we have
	\[
	\rho_{\Mh}(\tilde z^h_{j+1}, z^h_{j+1})\leq  \norm{\tilde z^h_{j+1}- z^h_{j+1}} + c_{\Mh}\cdot\norm{\tilde z^h_{j+1}- z^h_{j+1}}^2  
	,\]
	where $ c_{\Mh} $ is a constant depending on $ rch(\Mh) $, the manifold's reach.
	Explicitly, the reach bounds the maximal principal curvature of $ \Mh $, and as such it bounds the second order directional derivatives of the manifold as a local function over the tangent.  
	Therefore, we get
	\[
	\rho_{\Mh}(\tilde z^h_{j+1}, z^h_{j+1}) \leq C h^k
	,\]
	for some constant $ C $, as required.
\end{proof}

\section{Numerical Simulations}\label{sec:Numerical}

We now propose a simple algorithm to approximate geodesic distances based upon the aforementioned theoretical analysis.
The basic idea is to sample $ \Mh $ with many more samples than in the original point cloud to arrive at accurate approximations of geodesic distances on $ \Mh $.

\subsection{Naive Geodesic Approximation}\label{sec:Algorithms}
Let $ R = \{r_i\}_{i=1}^n $ be a sample-set of $\MM $ a smooth $ d $-dimensional submanifold of $ \RR^D $.
Then, we create $ X = \{x_\ell\}_{\ell=1}^N $ ($ N = K^d\cdot n $), a new set of samples of $ \Mh $ by means of Algorithm \ref{alg:enrich}.
Explicitly, for each point in the original set $ r_i $ find $ (q(r_i), H(r_i)) $ of \eqref{eq:Step1Minimization} (for an implementation of this see \cite{sober2019manifold,sober2017approximation}).
We then create a uniform grid $ \{\tilde x_{ij}\}_{j=1}^{K^d} $ on $ H $ around the origin.
In other words, in each basis direction on $ H $ take $ K $ uniform samples from $ -\sigma $ to $ \sigma $.
Then, find $ x_{ij} $ for $j=1,\ldots,{K^d} $ by computing the $ k ^{\textrm{th}}$ degree Manifold-MLS projection of $ \tilde x_{ij} $ onto $ \Mh $; i.e.,
\[
x_{ij} = \PP_k(\tilde x_{ij})
.\]

\begin{algorithm}
	\caption{Re-sampling of $ \Mh $}
	\label{alg:enrich}
	\begin{algorithmic}[1]
		\State{\bfseries Input:} \begin{tabular}{ll}
			$ R = \lbrace r_i \rbrace_{i=1}^n\subset \MM\subset\RR^D $ & the original sample set \\
			$ d $ & the intrinsic dimension of the manifold \\
			$ K $ & an enlarging factor\\
			$ k $ & local polynomial degree \\
			$ \sigma $ & parameter to localize the grid - choose $ \sigma \propto h $
		\end{tabular}
		\State{\bfseries Output:} $X =\{x_j\}_{\ell=1}^{K^d\cdot n} \subset \Mh\subset\RR^D$ - a new sample set of $ \Mh $\\
		\For{$r_i \in R$}
			\State Find $ (q(r_i), H(r_i)) $ of \eqref{eq:Step1Minimization} \Comment{see \cite{sober2019manifold,sober2017approximation} for implementation details}
			\State Create $ \{\tilde{x}_{ij}\}_{j=1}^{K^d} $ a uniform grid on $ H $ \Comment{sample between $ [-\sigma,\sigma] $ in each direction.}
			\State For $ j=i\cdot K + j $ project $ x_\ell = \PP_k(\tilde{x}_{ij}) $ \Comment{see \cite{sober2019manifold,sober2017approximation} for implemetation details}
		\EndFor
	\end{algorithmic}
\end{algorithm}

\begin{figure}[h]
	\centering
	\includegraphics[width=\linewidth]{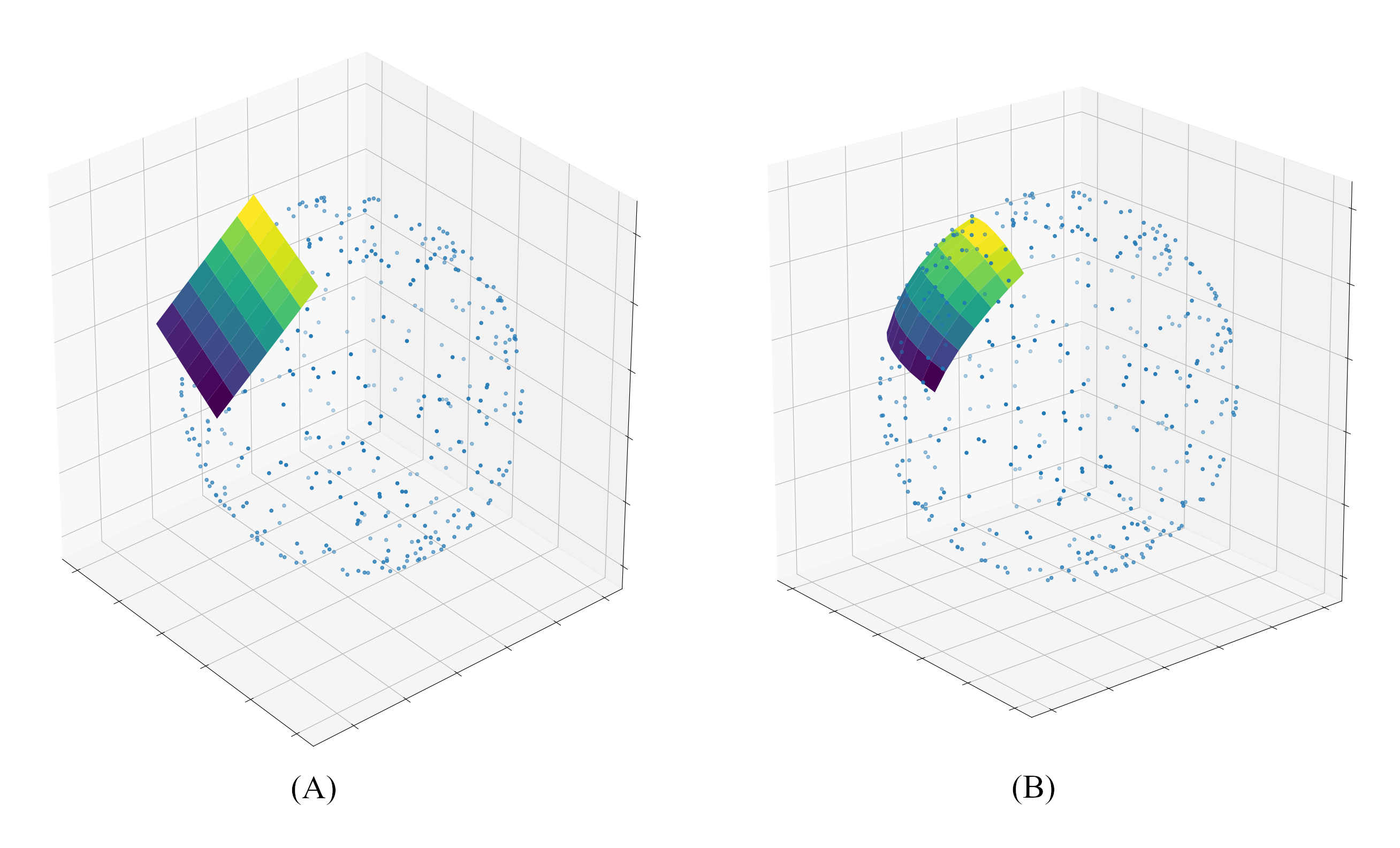}
	\caption{Projecting a grid on the tangent onto the manifold. (A) The original sample set and a grid defined on the tangent approximation produced by the Manifold-MLS approximation. (B) The Manifols-MLS projection of the tangential grid.}
	\label{fig:ProjectTangent}
\end{figure}

Finally, after obtaining the new sample set $ X $, given $ p_1, p_2 \in \MM $ we compute
\[
p_1^h = \PP_k(p_1),~ p_2^h = \PP_k(p_2)
,\]
and compute the shortest path between $ p_1^h $ and $ p_2^h $ with respect to the new cloud of points $ X $.

In the experiments below we have pre-computed $ \sigma $ by the heuristic approach presented in Algorithm \ref{alg:sigma}.

\begin{algorithm}
	\caption{Compute $ \sigma $}
	\label{alg:sigma}
	\begin{algorithmic}[1]
		\State{\bfseries Input:} \begin{tabular}{ll}
			$ R = \lbrace r_i \rbrace_{i=1}^n\subset \MM\subset\RR^D $ & the original sample set \\
			$ d $ & the intrinsic dimension of the manifold \\
			$ k $ & local polynomial degree 
		\end{tabular}
		\State{\bfseries Output:} $\sigma$ \\
		\State $ \sigma_\textrm{est} = zeros(1,100) $
		\For{$j = 1,\ldots,100$}
			\State Pick $ r_{i'}\in R $ at random
			\State Find $ \sigma_{i'} $ such that $ \# (B(r_{i'},\sigma_{i'})\cap R) =  {{k + d} \choose k }$
			\State $ \sigma_\textrm{est}(j) =  \sigma_{i'}$
		\EndFor
		\State $\sigma \leftarrow \max(\sigma_\textrm{est})$ 
	\end{algorithmic}
\end{algorithm}

\subsection{Qualitative Examples}
In the first qualitative experiment, we have drawn uniformly $ 20^2 $ samples of a 2-sphere embedded in $ \RR^3 $ (Fig. \ref{fig:DiscreteShortestPath}B). 
Using Algorithms \ref{alg:enrich} and \ref{alg:sigma}, with the parameters $ k = 3, K = 5, d=2 $ we have created a sample set $ X\subset\Mh $ of cardinality $ 20^2\times 5^2 $.
Fig. \ref{fig:DiscreteShortestPath} shows a comparison between the shortest path computed on $ R $ compared with the shortest path computed on $ X $ (both paths were computed by Dijkstra algorithm).
It is apparent that the initial sample set $ R $ has insufficient resolution, and thus provides very poor estimation of the geodesic path, whereas the generated sample $ X $ yields a very accurate representation of the geodesic using the same data.
\begin{figure}
	\centering
	\includegraphics[width=\linewidth]{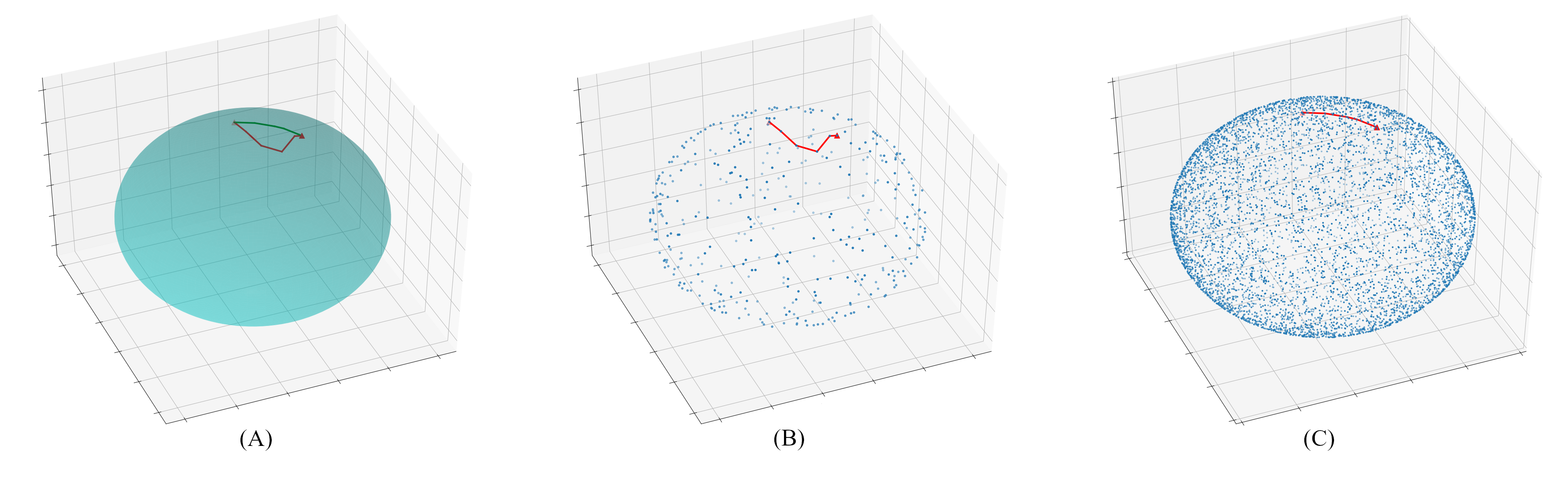}
	\caption{Shortest path from discrete samples of a 2-sphere. (A) The original manifold $ \MM $; the shortest path computed between $ p_1, p_2 $ (marked in red triangles) on discrete samples of the manifold (the path is marked by the red line); shortest path on re-sampling of the Manifold-MLS approximation to $ \MM $ (the path is marked by the green line).
		(B) The original samples of $ \MM $ and the shortest path connecting points $ p_1, p_2 $ marked by the red triangles.
		(C) The samples of the approximating manifold and the shortest path connecting points $ p_1, p_2 $ marked by the red triangles. }
	\label{fig:DiscreteShortestPath}
\end{figure}

On the second qualitative experiment, we took a point cloud obtained from the digitized talus bone of an \emph{Aotus trivirgatus}, a nocturnal monkey native to South America. The data was obtained from Morphosource\footnote{\url{https://www.morphosource.org/Detail/SpecimenDetail/Show/specimen_id/549}} and subsequently processed in Meshlab \cite{cignoni2008meshlab}. (Fig. \ref{fig:BoneShortestPath}C-D). 
In this experiment we first sub-sample the manifold, then compute the shortest path (Fig. \ref{fig:BoneShortestPath}A). Then we up-sample the manifold using the Manifold-MLS and compute the shortest path again (Fig. \ref{fig:BoneShortestPath}B).
\begin{figure}
	\centering
	\includegraphics[width=0.9\linewidth]{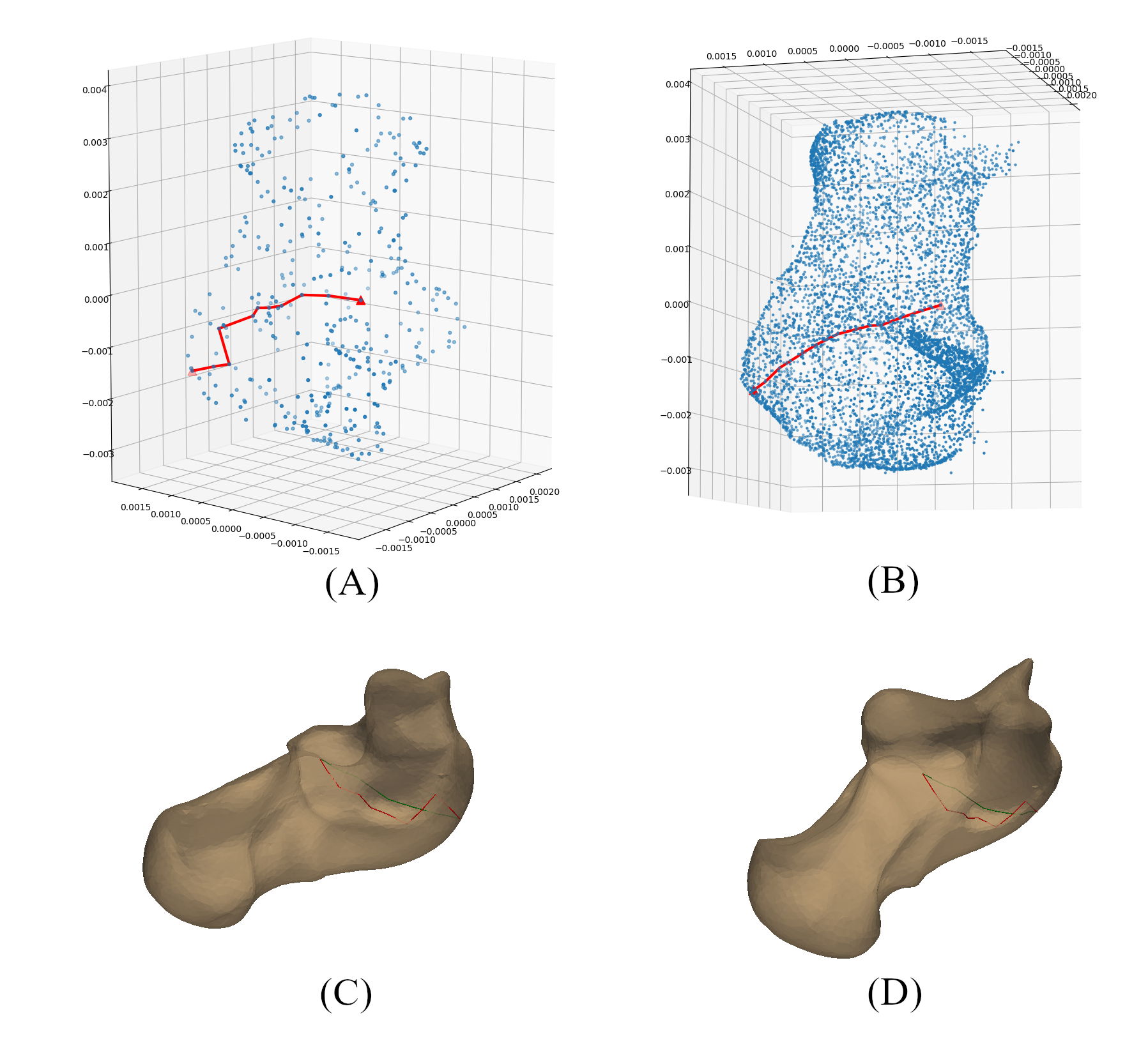}
	\caption{Shortest path from discrete samples of a talus bone. (A) A down-sampled Point Cloud of a talus bone and the shortest path computed between $ p_1, p_2 $ (marked in red triangles) on the sub-sampled manifold (the path is marked by the red line)
	(B) The samples of the $\Mh$ and the shortest path connecting points $ p_1, p_2 $ marked by the red triangles. 
	(C-D) Two angles of both paths overlaying the original mesh of the talus bone. The path computed on the sparse sample set (red) and the path computed on $\Mh$ (green).}
	\label{fig:BoneShortestPath}
\end{figure}

\subsection{A Quantitative Experiment}
In this experiment we aimed at performing a large scale comparison between the accuracy of estimating geodesic distance based upon shortest paths on the original samples of $ \MM $ with that of the approximating manifold.
Although in both cases we compute the geodesic estimation by means of the Dijkstra algorithm, using the generated sample $ X $ is assumed to be a close estimation of the geodesic distance of $ \Mh $, since we can resample $ \Mh $ using Algorithm \ref{alg:enrich} as fine as we wish.
Nonetheless, as can be seen in Table \ref{tab:accuracy} in order to provide good estimation to the geodesic distances the enlarging factor $ K $ need not be to large (in all the experiments below we have used $ K=3 $).

The experiment was conducted on random $ d $-spheres of radius $ 0.5 $ embedded in the cube $ [0,1]^{20} \subset \RR^{20}$, with varying intrinsic dimension and $ \#R = 10^d $. 
The error was computed by the relative root mean squared error (RMSE\%)
\[
\textrm{RMSE\%} = \frac{1}{\textrm{AVG}}\sqrt{\frac{1}{n} \sum_{i=1}^{n} \norm{x_i - \hat x_i}^2}
,\]
where $ \textrm{AVG} $ is the average of the accurate values.
The RMSE\% was computed over 100 realizations.
Explicitly, we have generated a random sample and then measured 100 distances between pairs of samples.
We compare the results of performing Dijkstra on the original sample $ R $, a generated sample $ X_1 $ using $ 1 $-degree Manifold-MLS, and a generated sample $ X_3 $ using $ 3 $-degree Manifold-MLS.
\begin{table}
	\begin{center}
		\begin{tabular}{|c|c|c|c|c|c|}
			\hline
			$ d $  & $ n $ (samples) & $ R $ (RMSE\%) & $ X_1 $ (RMSE\%) & $ X_3 $ (RMSE\%) & Euclidean (RMSE\%) \\ 
			\hline\hline
			2 & 100 & 17.3\% & 8.7\% & 0.5\% & 24.3\%\\
			\hline		
			3 & 1000 & 36.3\% & 1.1\% & 3.2\% & 18\%\\
			\hline
			4 & 10000 & 49.1\% & 7.4\% & 8.0\% & 16\%\\
			\hline				
		\end{tabular}

		\caption{Accuracy of geodesic distance estimation. 
			The experiment was conducted on randomly chosen $ d $-dimensional spheres of radius $ 0.5 $ embedded in $ [0,1]^{20}\subset\RR^{20} $, for which $ n $ uniformly distributed samples were taken.
			Column $ R $ represents the error of applying the Dijkstra algorithm on the original set $ R $.
			Column $ X_1, X_3 $ show the results of applying the Dijkstra algorithm on newly sampled set using $ K=3 $ and degrees $ k=1,3 $ of Manifold-MLS respectively.
			Column Euclidean, represents the error of computing simple Euclidean distance between the points in $ \RR^{20} $.
			The root mean squared error was computed over 100 realizations.}  
		\label{tab:accuracy}
	\end{center}    
\end{table}

\subsection{Robustness to Noise}
Here we took 2-dimensional random sphere of radios $ 0.5 $ embedded in the cube $ [0,1]^{20}\subset \RR^{20} $ and drew 400 samples uniformly distributed on the sphere with additive Gaussian noise $ \epsilon_i\in \RR^{20}$ and $ \epsilon_i \sim \NN(0, \sigma) $ for $ \sigma = 10^{-5}, 10^{-3}, 10^{-2} $.
Explicitly, the sample set was $ R = \{r_i\}_{i=1}^{400} $, and $ r_i = p_i + \epsilon_i $, where $ p_i \in \MM $ and $ \epsilon_i $ is as described above.

We then measured the geodesic distances between $ 100 $ pairs of randomly chosen points from the sample set.
The geodesic distances were measured through applying Dijkstra on the original sample $ R $, a generated sample $ X_1 $ using a $ 1 $-degree Manifold-MLS, and a generated sample $ X_3 $ using a $ 3 $-degree Manifold-MLS.
The generated samples were created by means of Algorithm \ref{alg:enrich} using $ K=5 $.
The results of this experiment are summarized in Table \ref{tab:robust} and show that our algorithm is robust to high levels of noise. 

\begin{table}
	\begin{center}
		\begin{tabular}{|c|c|c|c|c|}
			\hline
			Noise Level  & $ R $ (RMSE\%) & $ X_1 $ (RMSE\%) & $ X_3 $ (RMSE\%) & Euclidean (RMSE\%) \\ 
			\hline\hline
			$ 10^{-5} $ & 17.4\% & 2.93\% & 0.36\% & 15.416\%\\
			\hline		
			$ 10^{-3} $ & 16.0\% & 2.93\% & 0.41\% & 15.413\%\\
			\hline
			$ 10^{-2} $ & 30.0\% & 3.20\% & 1.83\% & 15.282\%\\
			\hline				
		\end{tabular}
		
		\caption{Accuracy of geodesic distance estimation. 
			The experiment was conducted on randomly chosen $ 2 $-dimensional spheres of radius $ 0.5 $ embedded in $ [0,1]^{20}\subset\RR^{20} $, for which $ 400 $ uniformly distributed samples were taken.
			The first column represents the standard deviation of the additive Gaussian noise $ \epsilon_i $.
			Column $ R $ represents the error of applying the Dijkstra algorithm on the original set $ R $.
			Column $ X_1, X_3 $ show the results of applying the Dijkstra algorithm on newly sampled set using $ K=5 $ and degrees $ k=1,3 $ of Manifold-MLS respectively.
			Column Euclidean, represents the error of computing simple Euclidean distance between the points in $ \RR^{20} $.
			The root mean squared error was computed over 100 realizations.}   
		\label{tab:robust}
	\end{center}    
\end{table}
\section*{Acknowledgments}

\bibliographystyle{abbrv}
\bibliography{GeodesicsApprox}

\appendix
\section{\hrho Sets, Quasi Uniform Sampling and Interior Cone Condition}\label{sec:Appendix}
The purpose of this appendix is to bridge between two different sets of assumptions that are used in the convergence analysis of MLS-based methods.
We wrote this appendix in a standalone manner, so that scholars interested just in this explanation will not need to read the whole paper.
Therefore, we repeat some of the definitions explained in Section \ref{sec:Preliminaries}.
As mentioned above, Levin introduced in 1998 the notion of \hrho sets (see Definition \ref{def:h-rho-delta} below) to show optimal rates of convergence of the MLS for function approximation \cite{levin1998approximation}.
In his paper, Levin handles the case of unbounded and boundaryless flat domains.
Then, in order to utilize Levin's bounds, this definition reappeared in \cite{sober2019manifold,sober2017approximation} where similar bounds are derived for the Manifold-MLS as well as MLS approximation of functions defined over manifold domains.
These works lay the foundations for the current paper.

However, in 2001, Wendland takes a seemingly different approach to achieve similar error analysis \cite{wendland2001local,wendland2004scattered} which reappears in \cite{mirzaei2015analysis}, whose results enables a crucial step in the analysis presented in the current paper. 
Namely, Wendland introduces the notion of quasi-uniform sample sets and takes a more general approach to show convergence rates for any \textit{polynomial reproduction} formula. 
Furthermore, in order to deal with boundary problems, Wendland requires that the domain maintain the Interior Cone Condition, which is the case with unbounded domains. 

In the following passages we show that Levin has some redundancy in his \hrho set definition, and that, essentially, any quasi-uniform set is \hrho set.
Thus, the results obtained in \cite{sober2019manifold} and \cite{sober2017approximation} could have been obtained using the assumption of a quasi-uniform set.

\subsection{\hrho Sets and Convergence Rates}
Let $ X = \{x_i\}_{i=1}^n $ be a sample set with pairwise distinct points in $ \Omega\subset\RR^d $, and let $ f(x_i) $ be the function evaluation at this point.
Then, by choosing $ u_j(x):\RR^d\to\RR $ we can construct an approximation of the form
\[s(x) = \sum_{j=1}^n u_j(x)f(x_j).\]
Such an approximation is sometimes called quasi-interpolant.

The MLS approximation is defined through the following minimization.
Let $ \pi^*(x ~|~ \xi) $ be defined by
\begin{equation}
	\pi^*(x ~|~ \xi) = \argmin_{\pi\in\Pi_{k-1}^d}\sum_{i=1}^n\norm{f(x_i) - \pi(x_i)}^2\theta_h(\norm{\xi - x_i})
,\end{equation}
where $ \theta_h $ is compactly supported on $[0, qh]$ and is consistent across scales (i.e., $ \theta_h(th) = \Phi(t) $), and, $ \norm{~\cdot~} $ denotes the Euclidean norm.
Then we define the MLS function as
\begin{equation}\label{eq:MLS_basic}
s^{\text{MLS}}_{f,X}(x) \defeq \pi^*(0 ~|~ x) 
.
\end{equation}

Levin showed in \cite{levin1998approximation} that, given enough samples in $ I(x) $, the neighborhood of $ x $, the MLS approximation defined in  \eqref{eq:MLS_basic} can be pronounced as
\[
\sum_{i\in I(x)} a^*_i(x) f(x_i)
,\]
where the coefficients $ a^*_i(x) $ are determined by minimizing the quadratic form
\[
\sum_{i\in I(x)}a_i(x)^2\frac{1}{\theta_h(x-x_i)}
\] 
under the constraints 
\[
\sum_{i\in I(x)}a_i(x) p(x_i) = p(x),~~ \forall p\in \Pi_{k-1}(\RR^d)
.\]

In other words, the MLS approximant can be  written as the following quasi-interpolant
\begin{equation}\label{eq:s_fX_MLS}
	s^{\text{MLS}}_{f,X}(x) = \sum_{i\in I(x)} a^*_i(x) f(x_i) = \pi^*(0 ~|~ x) 
	.
\end{equation}
In order to show the rates of convergence for the MLS, Levin uses the following notion of an \hrho set.

\begin{definition}[\hrho sets of fill distance $h$, density $\leq \rho$, and separation $\geq \delta$]\label{def:h-rho-delta}
	Let $\Omega$ be a $d$-dimensional domain in $\RR^n$, and consider sets of data points in $\Omega$. We say that the set $X = \lbrace x_i \rbrace_{i=1}^n$ is an  \hrho set if:
	\begin{enumerate}
		\item $h$ is the fill distance with respect to the domain $\Omega$
		\begin{equation*}
		h_{X,\Omega} = \sup_{x\in\Omega} \min_{x_i \in X} \norm{x - x_i}
		.
		\end{equation*}
		Throughout the appendix we will denote for short $ h:=h_{X,\Omega} $.
		\item 
		\begin{equation*}
		\#\left\lbrace X \cap \overline{B}_{qh}(y)  \right\rbrace \leq \rho \cdot q^d, ~~ q\geq 1, ~~ y \in \RR^n.
		\end{equation*}
		Here $\# Y$ denotes the number of elements in a given set $Y$, while $\overline{B}_r(x)$ is the closed ball of radius $r$ around $x$.
		\item $\exists \delta>0$ such that
		\begin{equation*}
		\norm{x_i - x_j} \geq h \delta, ~~ 1 \leq i < j \leq n.
		\end{equation*}
	\end{enumerate}
	\label{def:h-rho-delta_appendix}
\end{definition}
\begin{theorem}[Theorem 5 from \cite{levin1998approximation}]\label{thm:LevinApproximation}
	Let $ f\in \C^k(\Omega) $, and let $ \Omega $ be a compact domain. Then for fixed $ \rho $ and $ \delta $, there exists a fixed $ q>0 $, independent of $ h $, such that $ s_{f,X}^{\text{MLS}} $ of \eqref{eq:s_fX_MLS} with compactly supported weight function $ \theta $ with support of size $ qh $, satisfies
	\[
	\norm{s_{f,X}^{\text{MLS}} - f(x)}_{\Omega, \infty} \leq C_f \cdot h^k
	,\]
	for $ h $ sufficiently small, for \hrho sets of data points, and the constant $ C_f $ depends only on $ f $. 
\end{theorem}

The proof of this theorem relies on the the following observations.
Let $ p\in\Pi_{k-1} $, then
\begin{align*}
\abs{s_{f,X}^{\text{MLS}}(x) - f(x)} \leq \abs{s_{f,X}^{\text{MLS}} - p(x)} + \abs{p(x) - f(x)}
.\end{align*}
Thus,
\begin{align*}
\abs{s_{f,X}^{\text{MLS}}(x) - f(x)} &\leq \abs{\sum_{i\in I(x)}a^*_i(x)(f(x_i) - p(x_i))} + \norm{p(x) - f(x)}_{\Omega, \infty} \\
&\leq \sum_{i\in I(x)}\abs{a^*_i(x)}\abs{f(x_i) - p(x_i)} + \norm{p(x) - f(x)}_{\Omega, \infty}\\
&\leq \left(1 + \sum_{i\in I(x)}\abs{a^*_i(x)}\right)\norm{p(x) - f(x)}_{\Omega, \infty}
.\end{align*}
Since we can choose $ p(x) $ to be any polynomial in $ \Pi_{k-1} $ we can choose it to be the Taylor expansion of $f$ at $x$, so
\begin{equation}\label{eq:LevinApproximation}
\abs{s_{f,X}^{\text{MLS}}(x) - f(x)} \leq \left(1 + \sum_{i\in I(x)}\abs{a^*_i(x)}\right)h^k\abs{f}_{\C^k(\Omega)}
,\end{equation}
where $ \abs{f}_{\C^k(\Omega)} $ is the semi-norm defined by
$$ \abs{f}_{\C^k(\Omega)}\defeq \max_{\abs{\alpha} = k}\norm{\partial^\alpha f}_{\Omega, \infty} .$$
As a result, the key argument in providing the desired bound is through bounding the expression
\begin{equation}\label{eq:SumAi}
\sum_{i\in I(x)}\abs{a^*_i(x)}
.\end{equation}
Note, that due to the density parameter $ \rho $ of Definition \ref{def:h-rho-delta_appendix} the number of elements in $ I(x) $ the $ qh $-size neighborhood of $ x $ is bounded by $ N = \rho \cdot q^d $, which is independent of $ h $.
In addition, Levin shows that $ a_i^*(x) $ are continuous with respect to the installation of their neighboring samples $ x_j $.
Thus, by denoting the neighbors of $ x $ by $ \{x_{i1}, \ldots, x_{iN}\} $, we can write
\[
\eta(x_{i1}, \ldots, x_{iN}) = \abs{a_i^*(x)}
,\]
and
\[
C_{x, h} = \argmax_{(x_{i1}, \ldots, x_{iN})\in \overline{B}_{qh}(x)} \eta(x_{i1}, \ldots, x_{iN})
.\]
Note, that the maximum exists since the search space is compact.
The fact that the search space is bounded comes directly from the fact that $ (x_{i1}, \ldots, x_{iN})\in \overline{B}_{qh}(x) $.
To show that it is closed, let $ \xi = (x_{i1}, \ldots, x_{iN})  $ be a limit of a sequence $ \xi_k = (x^k_{i1}, \ldots, x^k_{iN}) $ with separation parameter $ \delta $, then $ \xi $ will still have a separation parameter $ \delta $ and will still be contained in $ \overline{B}_{qh}(x_i) $.

Thus, for a fixed $ x $ and fixed $ h $ it follows that
\[
\sum_{i\in I(x)}\abs{a_i^*(x)} \leq  C_{x, h}
,\]
where $ C_{x, h} $ is the same for all \hrho sets of points.
Furthermore, since each $ a_i^*(x) $ depends only on the installation if its $ qh $-size neighboring points, it is clear that $ C_{x, h} = C_h $ is independent of $ x $. 
The proof is finalized by showing that if the weighting is consistent across scales (i.e., $ \theta_h(th)=\Phi(t) $), then $ a_i^*(x) $ are independent of $ h $.
Therefore, $ C_h = C $ is independent of $ h $ as well and
\[
\norm{s_{f,X}^{\text{MLS}} - f(x)}_{\Omega, \infty} \leq C\abs{f}_{\C^k(\Omega)}\cdot h^k
= C_f \cdot h^k
.
\]
\qed

Interestingly, Wendland shows that the definition of density is redundant through a packing argument; that is, it can be derived directly from the separation parameter $ \delta $.
Below we portray the argument to show that given a separation parameter $ \delta $ as above, there exists a density parameter $ \rho $ independent of $ h $.
\begin{lemma}\label{lem:h-delta-is-rho}
	Let $ X $ be a sample set with fill distance $ h $ and separation parameter $ \delta $ (see Definition \ref{def:h-rho-delta_appendix} above). 
	Then, there exists $ \rho $ independent of $ h $ such that 
	\begin{equation*}
	\#\left\lbrace X \cap \overline{B}_{qh}(y)  \right\rbrace \leq \rho \cdot q^d, ~~ q\geq 1, ~~ y \in \RR^n.
	\end{equation*}
	\begin{proof}
		Let $ I(x) $ denote the set of indices of samples $ X \cap \overline{B}_{qh}(x) $, we wish to show that there exists $ \rho $ such that $ \#I(x) \leq \rho q^d $, where $ q\geq 1 $.
		Note, that any open ball $ B_{\frac{1}{2}\delta h}(x_j) $ of radius $ \frac{1}{2}\delta h $ centered at $ x_j $ is disjoint to $ B_{\frac{1}{2}\delta h}(x_k) $ (for $ j,k\in I(x) $ and $ j\neq k $).
		That is,
		\[
			B_{\frac{1}{2}\delta h}(x_j) \cap B_{\frac{1}{2}\delta h}(x_k) = \emptyset
		.\]
		On the other hand, for all $ j\in I(x) $ 
		\[
			B_{\frac{1}{2}\delta h}(x_j) \subset B_{qh + \frac{1}{2}\delta h}(x)
		.\]
		Thus,
		\begin{align*}
		\#I\cdot \Vol{B_{\frac{1}{2}\delta h}(x_j)} &\leq \Vol{B_{qh + \frac{1}{2}\delta h}(x)}\\
		\#I\cdot \Vol{B_{1}(0)}\left(\frac{1}{2}\delta h\right)^d &\leq \Vol{B_{1}(0)}\left(qh + \frac{1}{2}\delta h\right)^d\\
		\#I &\leq \left(2\frac{q}{\delta} + 1\right)^d 
		.\end{align*}
		Now, if $ \delta < 2 $ we have
		\[
		\#I \leq q^d \underbrace{\left(\frac{3}{\delta}\right)^d}_{\rho}
		,\]
		and if $ \delta \geq 2 $ we have
		\[
		\leq q^d  \underbrace{3^d}_{\rho}
		.\]
	\end{proof}
\end{lemma}
\subsection{Polynomial Reproduction and Interior Cone Condition}

In a similar fashion, Wendland aims at proving similar convergence rates for every polynomial reproduction process.
\begin{definition}\label{def:PolynomialReproduction}
	A process that defines for every set $ X = \{x_i\}_{i=1}^n \subset \RR^d$ a family of functions $ u_i = u_i^X:\Omega\to\RR $, $ 1 \leq i \leq n $, provides a local polynomial reproduction of degree $ k $ on $ \Omega $ if there exists constants $ h_0, C_1, C_2 > 0 $ such that
	\begin{enumerate}
		\item $\sum_{i=1}^{n}u_i(x) p(x_i) = p(x)$ for all $ p\in\at{\Pi_k^d}{\Omega} $,
		\item $ \sum_{i=1}^{n}\abs{u_i(x)}\leq C_1 $ for all $ x\in\Omega $,
		\item $ u_i(x) = 0 $ if $ \norm{x - x_i} > C_2 h $ and $ x\in \Omega $
	\end{enumerate}  
is satisfied for all $ X $ with fill distance $ h \leq h_0 $.
\end{definition}

Using this definition, and applying similar considerations to Levin, he shows the following theorem.

\begin{theorem}[Theorem 3.2 from \cite{wendland2004scattered}]\label{thm:PolyReproduction}
	Suppose that $ \Omega\subset\RR^d $ is bounded.
	Define $ \Omega^* $ to be the closure of $ \cup_{x\in\Omega}B(x, C_2 h_0) $.
	Define $ s_{f,X} = \sum_{i=1}^{n} u_i f(x_i) $, where $ \{u_i\} $ is a local polynomial reproduction of order $ (k-1) $ on $ \Omega $.
	If $ f\in \C^{k}(\Omega^*) $ then there exists a constant $ C > 0 $ depending only on the constants from Definition \ref{def:PolynomialReproduction} such that
	\[
	\abs{f(x) - s_{f, X}(x)}\leq C h^{k}\abs{f}_{\C^k(\Omega^*)}
	\]
	for all $ X $ with $ h \leq h_0 $.
	The semi-norm on the right hand side is defined by 
	$$ \abs{f}_{\C^k(\Omega^*)}\defeq \max_{\abs{\alpha} = k}\norm{D^\alpha f}_{L_\infty(\Omega^*)} .$$
\end{theorem}
Wendland then proceeds to show the argument regarding the MLS function approximation being a polynomial reproduction formula.
Using the representation of the MLS approximation from \eqref{eq:s_fX_MLS}, and by introducing the definition of a quasi-uniform sample set, Wendland establishes the fact that $ s_{f,X}^{\text{MLS}}$ is indeed a polynomial reproduction of degrree $ (k-1) $ on a compact domain as detailed below.

\begin{definition}[quasi-uniform sample set]
	A set of data sites $ X = \{x_1,\ldots,x_n\} $ is said to be quasi-uniform with respect to a constant $ c_{\emph{qu}} > 0 $ if
	\[
	\delta_X \leq h_{X,\Omega} \leq c_{\emph{qu}}\delta_X
	,\]
	where $ h_{X,\Omega} $ is the fill distance of Definition \ref{def:h} and $ \delta_X $ is the separation radius defined by
	\[
	\delta_X := \frac{1}{2}\min_{i\neq j}\norm{x_i - x_j}
	.\] 
\end{definition}
Note that this definition coincides with demanding just $ h $ and $ \delta $ conditions of the \hrho set of Definition \ref{def:h-rho-delta_appendix} above.
Thus, by Lemma \ref{lem:h-delta-is-rho} we get
\begin{corollary}\label{cor:h-rho_is_quasi}
	\hrho sets are the same as quasi-uniform sample sets.
\end{corollary}
Then, using this notion of quasi-uniform sample set, Wendland shows the following result.
\begin{theorem}\label{thm:WendlandApproximation}
	Suppose $ \Omega\subset\RR^d $ is compact and satisfies the interior cone condition with angle $ \alpha\in (0,\pi/2) $ and radius $ r>0 $. Fix $ k\in \NN $. Then there exist constants $ h_0, C_1 $, and $ C_2 $ depending only on $ k, \alpha, r $ such that for a quasi-uniform sample set $ X = \{x_i\}_{i=1}^n \subset \Omega$ with $ h_{X,\Omega}\leq h_0 $ and $ \theta $ with support $ 2C_2 h_{X,\Omega} $,  the functions $ a^*_i(x) $ provide local polynomial reproduction.
\end{theorem}
It is clear that, due to Theorem \ref{thm:PolyReproduction} above, this implies that the MLS approximation has $ O(h^k) $ convergence rates. 
The only detail that still needs to be clarified in this theorem is the \emph{interior cone condition}.
\begin{definition}[interior cone condition]
	A set $ \Omega\subset\RR^d $ is said to satisfy the interior cone condition if there exists an angle $ \alpha\in (0, \pi/2) $ and a radius $ r>0 $ such that for every $ x\in\Omega $ a unit vector $ \xi(x) $ exists such that the cone
	\[
	C(x,\xi,\alpha,r) := \{x + \lambda y ~|~ y\in \RR^d,~ \norm{y} = 1,~ y^T\xi\geq\cos\alpha,~ \lambda\in [0,r]
	\}
	\]
	is contained in $ \Omega $.
\end{definition}
 The main difference between Levin's proof of Theorem \ref{thm:LevinApproximation} and Wendland's proof of Theorem \ref{thm:WendlandApproximation} lies in the way they bound the expression
 \[
 \sum_{i\in I(x)} \abs{a_i^*(x)}
 .\]
 Wendland takes a more generalized approach of showing that given a compact domain $ \Omega $ that satisfies the interior cone condition, every sample set $ X \subset \Omega$ has a polynomial reproduction formula (see Theorem 3.14 in \cite{wendland2004scattered}).
 Then he uses this result to show the desired bound (see the proof of Theorem 4.7 in \cite{wendland2004scattered}). 
 
 \subsection{Conclusion}
 The main conclusion that we can draw from this discussion is that in order to obtain the approximation results reported in \cite{levin1998approximation,sober2019manifold,sober2017approximation} it is sufficient to demand that the sample set is quasi-uniform.
 Furthermore, as can be seen from Levin's result, the approximation is guaranteed when the weight function is compactly supported with a support of size $ qh $, given that there exists at least $ {{k + d} \choose k} $ samples in the support.
 Furthermore, this is also true of the analysis performed in \cite{mirzaei2015analysis}, if we neglect the boundary (e.g., by taking a sub-domain $ \Omega $ such that its closure is in the interior of $ \Omega^* $ and restricting the error analysis to $ \Omega $).
 As mentioned previous, since are dealing with boundaryless domain, this subtle issue is irrelevant to the discussions in the main text.

\end{document}